\documentclass[11pt]{article}

\usepackage{amsmath,amsfonts,amsthm,amscd,amssymb,graphicx}

\numberwithin{equation}{section}

\newtheorem{theorem}{Theorem}[section]

\newtheorem{lemma}[theorem]{Lemma}

\newtheorem{proposition}[theorem]{Proposition}
\newtheorem{prop}[theorem]{Proposition}

\newtheorem{remark}[theorem]{Remark}
\newtheorem{definition}[theorem]{Definition}

\def\eps{\varepsilon }
\def\D{\partial }

\newcommand{\RR}{\mathbb{R}}
\newcommand{\cO}{\mathcal{O}}

\newcommand{\CC}{\mathbb{C}}

\newcommand{\ZZ}{{\mathbb Z}}

\def\beq{\begin{equation}}
\def\eeq{\end{equation}}
\def\bb1{{1\!\!1}}
\def\R{\mbox{Re }}
\def\I{\mbox{Im }}

\def\dz{\partial_z}
\def\dx{\partial_x}

\def\rit{{\Bbb R}}
\def\cit{{\Bbb C}}

\def\eps{\varepsilon}

\begin{document}

\title{Spectral instability of symmetric shear flows in a two-dimensional channel}

\author{Emmanuel Grenier\footnotemark[1]
 \and Yan Guo\footnotemark[2] \and Toan T. Nguyen\footnotemark[3]
}

\date\today

\maketitle

\begin{abstract}
This paper concerns spectral instability of shear flows in the incompressible Navier-Stokes equations with sufficiently large Reynolds number: $R\to \infty$. It is well-documented in the physical literature, going back to Heisenberg, C.C. Lin, Tollmien, Drazin and Reid, that generic plane shear profiles other than the linear Couette flow are linearly unstable for sufficiently large Reynolds number. In this work, we provide a complete mathematical proof of these physical results. In the case of a symmetric channel flow, our analysis gives exact unstable eigenvalues and eigenfunctions, showing that the solution could grow slowly at the rate of $e^{t/\sqrt {\alpha R}}$, where $\alpha$ is the small spatial frequency that remains between lower and upper marginal stability curves: $\alpha_\mathrm{low}(R) \approx R^{-1/7}$ and $\alpha_\mathrm{up}(R) \approx R^{-1/11}$. We introduce a new, operator-based approach, which avoids to deal with matching inner and outer asymptotic expansions, but instead involves a careful study of singularity in the critical layers by deriving pointwise bounds on the Green function of the corresponding Rayleigh and Airy operators.

\end{abstract}

\renewcommand{\thefootnote}{\fnsymbol{footnote}}

\footnotetext[1]{Equipe Projet Inria NUMED,
 INRIA Rh\^one Alpes, Unit\'e de Math\'ematiques Pures et Appliqu\'ees., 
 UMR 5669, CNRS et \'Ecole Normale Sup\'erieure de Lyon,
               46, all\'ee d'Italie, 69364 Lyon Cedex 07, France. Email: egrenier@umpa.ens-lyon.fr}

\footnotetext[2]{Division of Applied Mathematics, Brown University, 182 George street, Providence, RI 02912, USA. Email: Yan\underline{~}Guo@Brown.edu}

\footnotetext[3]{Department of Mathematics, Pennsylvania State University, State College, PA~16803, USA. Email: nguyen@math.psu.edu. }

\tableofcontents

  
\newpage

\section{Introduction}

Study of hydrodynamics stability and the inviscid limit  of viscous fluids is one of the most classical subjects in fluid dynamics, going back to the most prominent physicists including Lord Rayleigh, Orr, Sommerfeld, Heisenberg, among many others. It is documented in the physical literature (see, for instance, \cite{LinBook,Reid}) that laminar viscous fluids are unstable, or become turbulent, in a small viscosity or high Reynolds number limit. In particular, shear flows other than the linear Couette flow in a two-dimensional channel are linearly unstable for sufficiently large Reynolds numbers. In the present work, we provide a complete mathematical proof of these physical results in a channel.

Specifically, let $u_0= (U(z),0)^{tr}$ be a stationary plane shear flow in a two-dimensional channel: $(y,z) \in \RR\times [0,2]$; see Figure \ref{fig-shear}. We are interested in the linearization of the incompressible Navier-Stokes equations about the shear profile: 
\begin{subequations}
\begin{align}
v_t +   u_0 \cdot \nabla v + v \cdot \nabla u_0  + \nabla p &= \frac {1}{R} \Delta v  \label{NS1}
\\
\nabla \cdot v &= 0 \label{NS2}
\end{align}
\end{subequations}
posed on $\RR\times [0,2]$, together with the classical no-slip boundary conditions on the walls:
\begin{equation}\label{NS3}
v_{\vert_{z=0,2}} = 0.  
\end{equation}
Here $v$ denotes the usual velocity perturbation of the fluid, and $p$ denotes the corresponding pressure. Of interest is the Reynolds number $R$ sufficiently large, and whether the linearized problem is spectrally unstable: the existence of unstable modes of the form $(v,p) =  (e^{\lambda t} \tilde v(y,z), e^{\lambda t} \tilde  p(y,z))$ for some $\lambda$ with $\Re \lambda >0$.


The spectral problem is a very classical issue in fluid mechanics.
A huge literature is devoted to its detailed study. We in particular refer to
\cite{Reid, Schlichting} for the major works of Heisenberg, C.C. Lin, Tollmien, and Schlichting.
The studies began around 1930, motivated by the study of the boundary layer
around wings. In airplanes design, it is crucial to study the boundary layer
around the wing, and more precisely the transition between the laminar and turbulent
regimes, and even more crucial to predict the point where boundary layer
splits from the boundary. A large number of papers has been devoted to 
the estimation of the critical Rayleigh number of classical shear flows 
(plane Poiseuille flow, Blasius profile, exponential suction/blowing profile, among others).

   \begin{figure}
\centering\includegraphics[scale=.5]{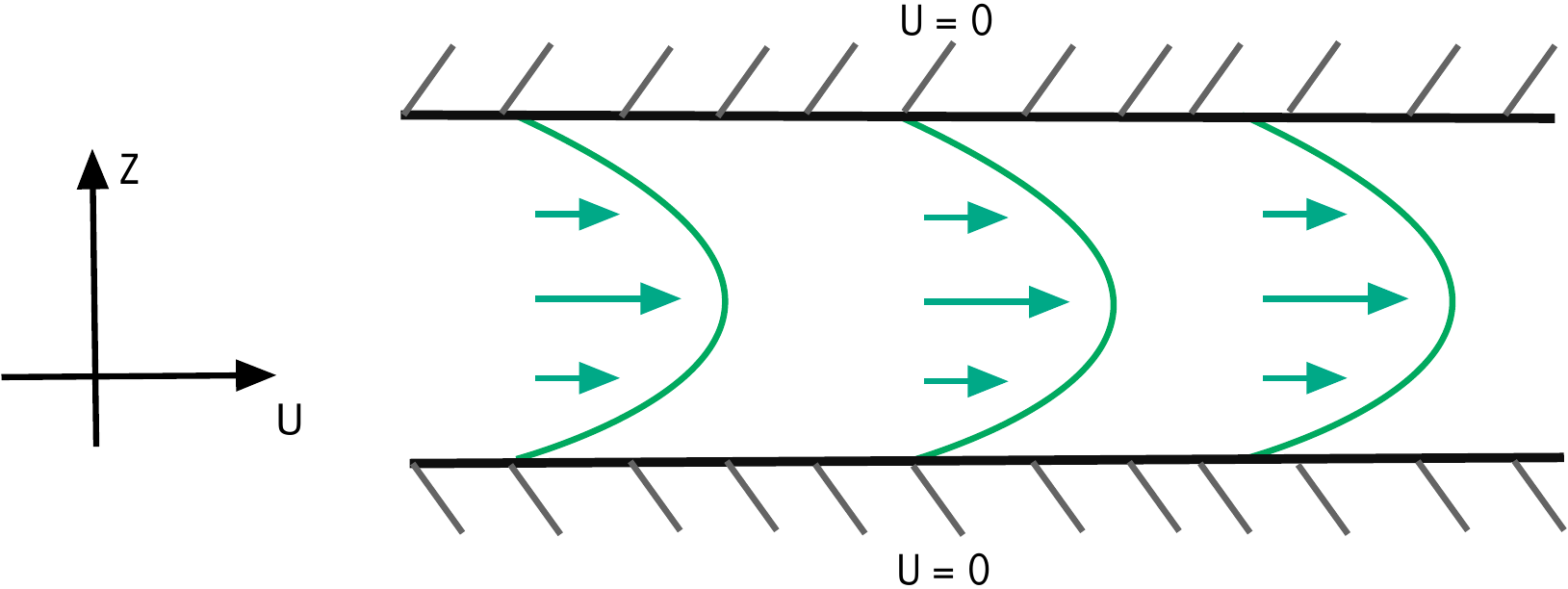}
\put(8,17){$z=0$}
\put(8,69){$z=2$}
\caption{\em Shown is the graph of an inviscid stable shear profile.}
\label{fig-shear}
\end{figure}

It were Sommerfeld and Orr \cite{Sommerfeld, Orr} who initiated the study of the spectral problem via the  Fourier normal mode theory. They search for the unstable solutions of the form $e^{i\alpha (y-ct)} (\hat v(z), \hat p(z))$, and derive the well-known Orr-Somerfeld equations for linearized viscous fluids: 
\beq \label{OS1-intro}
\epsilon (\dz^2 - \alpha^2)^2 \phi 
= (U-c) ( \dz^2 - \alpha^2) \phi  - U'' \phi, 
\eeq
with $\epsilon = 1/(i\alpha R)$, where $\phi(z)$ denotes the corresponding stream function, with $\phi$ and $\partial_z \phi$ vanishing at the boundaries $z = 0,2$. When $\epsilon = 0$, \eqref{OS1-intro} reduces to the classical Rayleigh equation, which corresponds to inviscid flows. The singular perturbation theory was developed to construct Orr-Somerfeld solutions from those of Rayleigh solutions. 

\bigskip

{\bf Inviscid unstable profiles.} If the profile is unstable for the Rayleigh equation, then there exist a spatial frequency $\alpha_\infty$, an eigenvalue $c_\infty$ with
$\I c_\infty > 0$, and a corresponding eigenvalue $\phi _\infty$ that solve \eqref{OS1-intro} with $\epsilon =0$ or $R = \infty$. 
We can then make a perturbative analysis to construct an unstable eigenmode $\phi _R$ of
the Orr-Sommerfeld equation with an eigenvalue $\I c_R >0$ for any large enough $R$. This can be done by adding a boundary sublayer to the inviscid mode $\phi_\infty$ to correct the boundary conditions for the viscous problem. In fact, we can further check that 
\beq \label{perturb1}
c_R = c_\infty + \cO(R^{-1}) ,
\eeq
as $R\to \infty$. Thus, the time growth is of order $e^{\theta_0 t}$, for some $\theta_0>0$. Such a perturbative argument for the inviscid unstable profiles is well-known; see, for instance, Grenier \cite{Gr1} where he rigorously establishes the nonlinear instability of inviscid unstable profiles.


\bigskip

{\bf Inviscid stable profiles.} There are various criteria to check whether a shear profile is stable to the Rayleigh equation. The most classical one was due to Rayleigh \cite{Ray}: {\em A necessary condition for instability is that $U(z)$ must have an inflection point}, or its refined version by Fjortoft \cite{Reid}:  {\em A necessary condition for instability is that $U'' (U - U(z_0))<0$ somewhere in the flow, where $z_0$ is a point at which $U''(z_0) =0$.} For instance, the plane Poiseuille flow: $U(z) = 1- (z-1)^2$, or the $\sin$ profile: $U(z) = \sin(\frac{\pi z}{2})$ are stable to the Rayleigh equation.

For such a stable profile, all the spectrum of the Rayleigh equation
is imbedded on the imaginary axis: $\R(-i\alpha c_\infty) = \alpha \I c_\infty = 0$, and thus it is not clear whether a perturbative argument to construct solutions $(c_R,\phi_R)$ to \eqref{OS1-intro} would yield stability ($\I c_R <0$) or instability ($\I c_R >0$). Except the case of the linear Couette flow $U(z) = z$, which is proved to be linearly stable for all Reynolds numbers by Romanov \cite{Rom}, {\em all other profiles (including those which are inviscid stable) are physically shown to be linearly unstable for large Reynolds numbers.} Heisenberg \cite{Hei,HeiICM} and then C. C. Lin \cite{Lin0,LinBook} were among the first physicists to use asymptotic expansions to study the instability; see also Drazin and Reid \cite{Reid} for a complete account of the physical literature on the subject. There, it is documented that there are lower and upper marginal stability branches $\alpha_\mathrm{low}(R), \alpha_\mathrm{up}(R)$ so that whenever $\alpha\in [\alpha_\mathrm{low}(R),\alpha_\mathrm{up}(R)]$, there exist an unstable eigenvalue $c_R$ and an eigenfunction $\phi_R(z)$ to the Orr-Sommerfeld problem.  In the case of symmetric Poiseuille profile: $U(z) = 1- (z-1)^2 $, the marginal stability curves are 
\begin{equation}\label{ranges-alpha0}\alpha_\mathrm{low}(R) = A_{1c}R^{-1/7}\qquad  \mbox{and}\qquad  \alpha_\mathrm{up}(R)= A_{2c} R^{-1/11},\end{equation}
for some critical constants $A_{1c}, A_{2c}$. Their formal analysis has been compared with modern numerical computations
and also with experiments, showing a very good agreement; see Figure \ref{fig-marginalcurves} or \cite[Figure 5.5]{Reid} for a sketch of the marginal stability curves for plane Poiseuille flow, which is an exact steady state solution to the Navier-Stokes equations.  

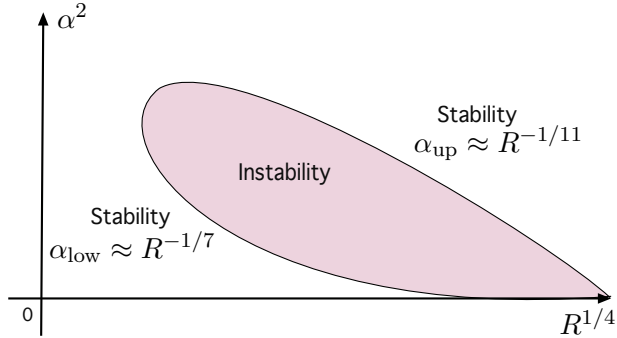
\begin{figure}[t]
\centering
\includegraphics[scale=.4]{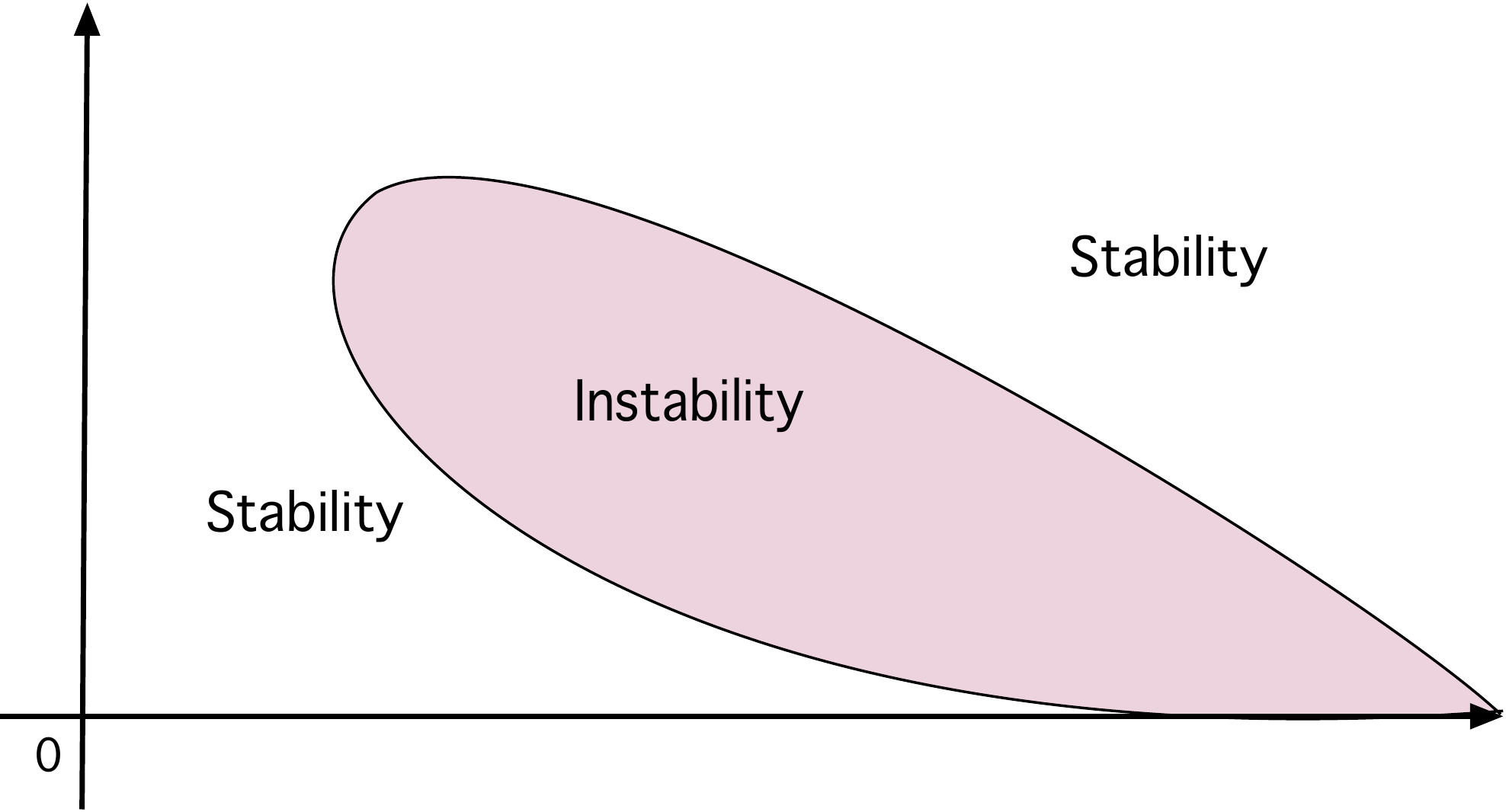}
\put(-20,1){$R^{1/4}$}
\put(-210,117){$\alpha^2$}
\put(-213,30){$\alpha_\mathrm{low}\approx R^{-1/7}$}
\put(-75,70){$\alpha_\mathrm{up}\approx R^{-1/11}$}
\caption{\em Illustrated are the marginal stability curves; see also \cite[Figure 5.5]{Reid}.}
\label{fig-marginalcurves}
\end{figure}


In his works \cite{W1,W3, Wbook}, Wasow developed the turning point theory to rigorously validate the formal asymptotic expansions used by the physicists in a full neighborhood of the turning points (or the critical layers in our present paper). It appears however that Wasow himself did not explicitly study how his approximate solutions depend on the three small parameters $\alpha, \epsilon,$ and $\I c$ in the Orr-Sommerfeld equations, nor apply his theory to resolve the stability problem (see his discussions on pages 868--870, \cite{W1}, or Chapter One, \cite{Wbook}). 

 Even though Drazin and Reid (\cite{Reid}) indeed provide many delicate asymptotic analysis in different regimes with different matching conditions near the critical layers, it is mathematically unclear how to combine their ``local'' analysis into a single convergent ``global expansion'' to produce an exact growing mode for the Orr-Sommerfeld equation. To our knowledge, remarkably, after all these efforts, a complete rigorous construction of an unstable growing mode is still elusive for such a fundamental problem.



Our present paper rigorously establishes the spectral instability of generic shear flows. The main theorem reads as follows. 


\begin{theorem}\label{theo-unstablemodes}
Let $U(z)$ be an arbitrary shear profile that is analytic and symmetric about $z=1$ with $U'(0) > 0$ and $U'(1) =0$. Let $\alpha_\mathrm{low}(R)$ and $\alpha_\mathrm{up}(R)$ be defined as in \eqref{ranges-alpha0}. Then, there is a critical Reynolds number $R_c$ so that for all $R\ge R_c$ and all $\alpha \in (\alpha_\mathrm{low}(R), \alpha_\mathrm{up}(R))$, there exist
a triple $c(R), \hat v(z; R), \hat p(z;R)$, with $\mathrm{Im} ~c(R) >0$, such that $v_R: = e^{i\alpha(y-ct) }\hat v(z;R)$ and $p_R: = e^{i\alpha(y-ct)} \hat p(z;R)$ solve the problem 
\eqref{NS1}-\eqref{NS2} with the no-slip boundary conditions. In the case of instability, there holds the following estimate for the growth rate of the unstable solutions:
$$ \alpha \I c(R) \quad \approx\quad  (\alpha R)^{-1/2},$$
as $R \to \infty$. In addition, the horizontal component of the unstable velocity $v_R$ is odd in $z$, whereas the vertical component is even in $z$. 
  
\end{theorem}

Theorem \ref{theo-unstablemodes} allows general shear profiles.
 The instability is found, even for inviscid stable flows such as plane Poiseuille flows,
  and thus is due to the presence of viscosity.
   It is worth noting that the growth rate is vanishing in the inviscid limit: 
   $R \to \infty$, which is expected as the Euler instability is necessary in the inviscid limit for the instability 
   with non-vanishing growth rate; for the latter result, see \cite{FV} in which general stationary profiles are considered. Linear to nonlinear instability is a delicate issue, primarily due to the fact that there is no available, comparable bound on the linearized solution operator as compared to the maximal growing mode. Available analyses (for instance, \cite{Fri, Gr1}) do not appear applicable in the inviscid limit. 




As mentioned earlier, we construct the unstable solutions via the Fourier normal mode method. Precisely, let us introduce the stream function $\psi$ through
\begin{equation}\label{def-stream}v = \nabla^\perp \psi = (\partial_z, -\partial_y)\psi 
,\qquad \psi(t,y,z) := \phi (z) e^{i \alpha (y - ct) },
\end{equation}
with $y\in \RR$, $z\in [0,2]$, the spatial frequency $\alpha \in \RR$ and the temporal eigenvalue $c\in \CC$. As our main interest is to study symmetric profiles, we will construct solutions that are also symmetric with respect to the line $z = 1$. The equation for vorticity $\omega = \Delta \psi$ becomes the classical Orr--Sommerfeld equation for $\phi$
\beq \label{OS1}
\epsilon (\dz^2 - \alpha^2)^2 \phi 
= (U-c) ( \dz^2 - \alpha^2) \phi  - U'' \phi  ,\qquad z\in [0,1], 
\eeq
with $\epsilon = {1 \over i \alpha R}.$ The no-slip boundary condition on $v$ then becomes 
\beq \label{OS2}
\alpha \phi  = \partial_z \phi  = 0 \quad\hbox{ at } \quad z = 0,
\eeq
whereas the symmetry about $z=1$ requires 
\beq \label{OS3}
\partial_z\phi  = \partial_z ^3\phi  = 0 \quad\hbox{ at } \quad z = 1.
\eeq
Clearly, if $\phi(z)$ solves the Orr-Sommerfeld problem  \eqref{OS1}-\eqref{OS3}, then the velocity $v$ defined as in \eqref{def-stream} solves the linearized Navier-Stokes problem with the pressure $p$ solving 
$$ -\Delta p = \nabla U \cdot \nabla v, \qquad \partial_z p _{\vert_{z=0,2}} = -\partial_z^2\partial_y \psi_{\vert_{z=0,2}}. $$
Throughout the paper, we study the Orr-Sommerfeld problem.



%

%

Delicacy in the construction is primarily due to the formation of  critical layers. To see this, let $(c_0, \phi_0)$ be a solution to the Rayleigh problem with $c_0\in \RR$. Let $z_0$ be the point at which 
\begin{equation}\label{cr-layer} U(z_0) = c_0.\end{equation}
Since the coefficient of the highest-order derivative in the Rayleigh equation vanishes at $z = z_0$, the Rayleigh solution $\phi_0(z)$ has a singularity of the form: $1+(z-z_0) \log(z-z_0)$. A perturbation analysis to construct an Orr-Sommerfeld solution $\phi_\epsilon$ out of $\phi_0$ will face a singular source $\epsilon (\partial_z^2 - \alpha^2)^2 \phi _0$ at $z = z_0$. To deal with the singularity, we need to introduce the critical layer $\phi_\mathrm{cr}$ that solves 
$$\epsilon \partial_z^4 \phi_\mathrm{cr} = (U-c)\partial_z^2 \phi_\mathrm{cr}$$
When $z$ is near $z_0$, $U - c$ is approximately $z-z_c$ with $z_c$ near $z_0$, and the above equation for the critical layer becomes the classical Airy equation for $\partial_z^2 \phi_\mathrm{cr}$. This shows that the critical layer mainly depends on the fast variable: $\phi_\mathrm{cr} = \phi_\mathrm{cr}(Y)$ with $Y = (z-z_c)/\epsilon^{1/3}$. 
 
In the literature, the point $z_c$ is occasionally referred to as a turning point, since the eigenvalues of the associated first-order ODE system cross at $z=z_c$ (or more precisely, at those which satisfy $U(z_c) = c$), and therefore it is delicate to construct asymptotic solutions that are analytic across different regions near the turning point. In his work, Wasow fixed the turning point to be zero, and were able to construct asymptotic solutions in a full neighborhood of the turning point. It is also interesting to point out that the authors in \cite{LWZ} recently revisit the analysis near turning points, and are able to construct unstable solutions in the context of gas dynamics, via WKB-type asymptotic techniques. 
   
In the present paper, we introduce a new, operator-based approach, which avoids dealing with inner and outer asymptotic expansions, but instead constructs the Green's function, and therefore the inverse, of the corresponding Rayleigh and Airy operators. The Green's function of the critical layer (Airy) equation is complicated by the fact that we have to deal with the second primitive Airy functions, not to mention that the argument $Y$ is complex. The basic principle of our construction, for instance, of a slow decaying solution, will be as follows. We start with an exact Rayleigh solution $\phi_0$ (solving \eqref{OS1} with $\epsilon =0$). This solution then solves \eqref{OS1} approximately up to the error term $\epsilon (\partial_z^2 - \alpha^2)^2 \phi _0$, which is singular at $z=z_0$ since $\phi_0$ is of the form $1+ (z-z_0)\log(z-z_0)$ inside the critical layer. We then correct $\phi_0$ by adding a critical layer profile $\phi_\mathrm{cr}$ constructed by convoluting the Green's function of the primitive Airy operator against the singular error $\epsilon (\partial_z^2 - \alpha^2)^2 \phi _0$. The resulting solution $\phi_0 + \phi_\mathrm{cr}$ solves \eqref{OS1} up to a smaller error that consists of no singularity. An exact slow mode of \eqref{OS1} is then constructed by inductively continuing this process. For a fast mode, we start the induction with a second primitive Airy function.

~\\
{\bf Notation.} 
Throughout the paper, the profile $U = U(z)$ is kept fixed. Let $c_0$ and $z_0$ be real numbers so that $U(z_0) = c_0$. We extend $U(z)$ analytically in a neighborhood of $z_0$ in $\CC$. We then let $c$ and $z_c$ be two complex numbers in the neighborhood of $(c_0,z_0)$ in $\mathbb{C}^2$ so that $U(z_c) = c$. 
 It follows by the analytic expansions of $U(z)$ near $z_0$ and $z_c$ that $|\I c| \approx |\I z_c | $, provided that $U'(z_0)  \not =0$. In the statement of the main theorem and throughout the paper, we take $z_0 = 0$.

~\\
{\bf Further notation.} We shall use $C_0$ to denote a universal constant that may change from line to line, but is independent of $\alpha$ and $R$. We also use the notation $f=\cO(g)$ or $f\lesssim g$ to mean that $|f|\le C_0 |g|$, for some constant $C_0$. Similarly, $f \approx g$ if and only if $f \lesssim g$ and $g \lesssim f$. Finally, when no confusion is possible, inequalities involved with complex numbers $|f| \le g$ are understood as $|f|\le |g|$.


\section{Strategy of proof}



\subsection{Operators}


For our convenience, let us introduce the following operators. Let us denote by  $Orr$ the Orr-Sommerfeld operator
\beq \label{opOrr}
Orr(\phi) := (U - c) (\dz^2 - \alpha^2) \phi - U'' \phi - \eps (\dz^2 - \alpha^2)^2 \phi,
\eeq
by $Ray_\alpha$ the Rayleigh operator
\beq \label{opRay}
Ray_\alpha(\phi): = (U-c) (\dz^2 - \alpha^2) \phi - U'' \phi,
\eeq
by $Diff$ the diffusive part of the Orr-Sommerfeld operator,
\beq \label{opDiff}
Diff(\phi) :=  - \eps (\dz^2 - \alpha^2)^2 \phi ,
\eeq
by $Airy$ the modified Airy equation
\beq \label{opAiry}
Airy(\phi) := \eps \dz^4 \phi - (U - c + 2 \eps \alpha^2) \dz^2 \phi ,
\eeq
and finally, by $Reg$ the regular zeroth order part of the Orr-Sommerfeld operator
\beq \label{opReg}
Reg(\phi): =- \Big[\eps \alpha^4 + U'' + \alpha^2 (U-c) \Big]\phi .
\eeq
Clearly, there hold identities
 \begin{equation}\label{key-ids}
Orr = Ray_\alpha + Diff = - Airy + Reg.
\end{equation}


\subsection{Outline of the construction}


Let us outline the strategy of the proof before going into the technical
details and computations. Our ultimate goal is to construct four independent solutions of the fourth order differential equation (\ref{OS1})
and then combine them in order to satisfy boundary conditions (\ref{OS2}) and
(\ref{OS3}), yielding the  linear dispersion relation. The unstable eigenvalues are then found by carefully studying the dispersion relation. 

The idea of the proof is to start from a mode of Rayleigh equation, or from an Airy function $\phi_0$. This
function is not an exact solutions of Orr Sommerfeld equations, but leads to an error
$$
E_0 = Orr(\phi_0).
$$
We correct it by adding $\phi_0^{Ray}$ defined by
$$
Ray_\alpha (\phi_0^{Ray}) = - Orr(\phi_0) .
$$
Again $\phi_0 + \phi_0^{Ray}$ is not an exact solution of Orr Sommerfeld equations
$$
Orr( \phi_0 + \phi_0^{Ray} ) = Diff (\phi_0^{Ray} ) .
$$
It turns out that, even if $\phi_0$ is smooth, $\phi_0^{Ray}$ is not smooth and contains a singularity
of the form $(z - z_c) \log (z -z _c)$. As a consequence, $Diff(\phi_0^{Ray})$ contains terms like
$1 / (z - z_c)^3$. To smooth out this singularity we use Airy operator and introduce
$\phi_0^A$ defined by
$$
Airy(\phi_0^A ) = - Diff (\phi_0^{Ray} ) .
$$
Then
$$
\phi_1 = \phi_0 + \phi_0^{Ray} + \phi_0^A 
$$
satisfies
$$
E_1 = Orr(\phi_1 ) = Reg(\phi_0^A) .
$$
Note that in some sense $\phi_0^A$ replaces the $(z - z_c) \log(z - z_c)$
singular term by a smoother one. 

We then iterate the construction. Note that
$$
E_1 = Reg \Bigl( Airy^{-1} \Bigl( Diff ( Ray^{-1} (E_0) ) \Bigr) \Bigr) .
$$
The main problem is to check the convergence of this process, and more precisely to prove that
$$
Reg \circ Airy^{-1} \circ Diff \circ Ray^{-1}
$$
has a norm strictly smaller than $1$ in suitable functional spaces.
Note that our approach avoids to deal with inner and outer expansions, but requires a careful
study of the singularities and delicate estimates on the resolvent solutions.

In the whole paper,
 $z_c$ is some complex number and will be fixed, depending only on $c$, through $U(z_c) = c$.

We introduce two families of function spaces, $X_p$ and $Y_p$ which turn out to be very
well fitted to describe functions which are singular near $z_c$.

First the the function spaces $X_p$ are defined by their norms:
$$
\| f\|_{X_p} : = \sup_{z\in [0,1] } \sum_{k=0}^p|  (z-z_c)^k \dz^k f(z) |.
$$ 
We also introduce the function spaces $Y_p$ defined by: 
$f \in Y_p$ if there exists a constant $C$ such that
$$
|f(z)| \le C \quad  \forall 0 \le z \le 1,
$$
$$
| \dz f(z) | \le C (1 + | \log (z - z_c) | )  \quad  \forall 0 \le z \le 1,
$$
$$
| \dz^l f(z) | \le C (1 + | z - z_c | ^{1 - l} ) 
$$
for every $0 \le z \le 1$ and every $l \le p$.
The best constant $C$ in the previous bounds is by definition the norm
$\| f \|_{Y_p}$.

Let us now sketch the key estimates of the paper. The first point is, thanks to almost
explicit computations, we can construct
an inverse operator $Ray^{-1}$ for $Ray_\alpha$. Note that if $Ray_{\alpha} (\phi) = f$,
then
\beq \label{Rayl}
(\partial_z^2 - \alpha^2) \phi = {U'' \over U - c} \phi + {f \over U - c} .
\eeq
Hence, provided $U-c$ does not vanish (which is the case when $c$ is complex), using classical
elliptic regularity we see that if $f \in C^k$ then $\phi \in C^{k+2}$. We thus gain two derivatives.
However the estimates on the derivatives degrade as $z - z_c$ goes smaller. The main point is
that the weight $(z-z_c)^l$ is enough to control this singularity. Moreover, deriving $l$ times 
(\ref{Rayl}) we see that $\partial_z^{2+l} \phi$ is bounded by $C / (z - z_c)^{l+1}$ if $f \in X_k$. Hence
we gain one $z - z_c$ factor in the derivative estimates between $f$ and $\phi$.
Hence if $f$ lies in $X_p$, $\phi$ lies in $Y_{p+2}$, with a gain of two derivatives and of an extra
$z - z_c$ weight. As a matter of fact we will construct an inverse $Ray^{-1}$ 
 which is continuous from $X_k$ to $Y_{k+2}$ for any $k$. 

\medskip

Using Airy functions, their double primitves, and a special variable and unknown transformation 
known in the literature as Langer transformation, we can construct an almost explicit inverse $Airy^{-1}$
to our $Airy$ operator.
We then have to investigate $Airy^{-1} \circ Diff$. Formally it is of order $0$, however it is singular,
hence to control it we need to use two derivatives, and to make it small we need a $z - z_c$ factor
in the norms. After tedious computations on almost explicit Green functions
 we prove that $Airy^{-1} \circ Diff$
has a small norm as an operator from
$Y_{k+2}$ to $X_k$. 

\medskip

Last, $Reg$ is bounded from $X_k$ to $X_k$, since it is a simple multiplication by a bounded function.
Combining all these estimates we are able to construct exact solutions of Orr Sommerfeld equations,
starting from solutions of Rayleigh equations of from Airy equations. This leads to the construction of
four independent solutions. Each such solution is defined as a convergent serie, which gives
its expansion. It then remains to combine all the various terms of all these solutions to get the dispersion
relation of Orr Sommerfeld. The careful analysis of this dispersion relation gives our instability result.

The plan of the paper follows the previous lines.

\newpage 

\section{Rayleigh equation}\label{sec-Rayleigh}


In this part, we shall construct an exact inverse for the Rayleigh operator $Ray_\alpha$
for small $\alpha$ and so find a complete solution to 
\begin{equation}\label{eq-Raya}
Ray_\alpha(\phi) = (U-c)(\dz^2 - \alpha^2)\phi - U'' \phi = f
\end{equation}
in accordance with the boundary condition: $\partial_z \phi_{\vert_{z=1}} = 0$. 
Note that as we do not prescribe a boundary condition at $z = 0$ there is not one unique inverse
for $Ray_\alpha$. We only construct one possible inverse.
To do so, we first invert the Rayleigh operator $Ray_0$ when $\alpha =0$ 
by exhibiting an explicit Green function. 
We then use this inverse to inductively construct the Green function
 and the inverse of the $Ray_\alpha$ operator.  Precisely, we will prove in this section the following Proposition. 

\begin{proposition}\label{prop-exactRayS} 
Assume that $\I c \not =0$ and that $\alpha$ is sufficiently small. 
Then, there exists a bounded operator $RaySolver_{\alpha} (\cdot) $
 so that 
\begin{equation}\label{eqs-RaySolver}
\begin{aligned}
 Ray_\alpha (RaySolver_{\alpha} (f)) (z) &= f(z), \qquad \forall ~z\in [0,1]
 \\
 \partial_z RaySolver_{\alpha} (f) _{\vert_{z=1}} &= 0
 \end{aligned}
 \end{equation} 
Morevover this operator is bounded from $X_k$ to $Y_{k+2}$
for every interger
$k \ge 0$, with
$$\| RaySolver_{\alpha}(f) \|_{Y_{k+2}} \le C_0 
\|f\|_{X_k},
$$ 
for some universal constants $C_k$. 

\begin{remark}{\em 
 If we assume further in Proposition \ref{prop-exactRayS} that $f'(1) = 0$,
  the equation \eqref{eqs-RaySolver} yields 
$$\partial^3_z RaySolver_{\alpha} (f) _{\vert_{z=1}} = 0.$$
This implies that the inviscid solution $RaySolver_{\alpha} (f)$ automatically satisfies the boundary condition
 \eqref{OS3} at $z=1$, and hence no boundary layer correctors are needed in vicinity 
 of the boundary point $z=1$.  
}
\end{remark}

\begin{remark}{\em
Away from $z_c$, Rayleigh equation is elliptic, hence it is natural two gain the control on two derivatives. Near
$z_c$, $\partial_z^l \phi$ behaves like $\partial_z^{l-2} (f / (z - z_c))$ if $l \ge 2$, which is coherent with the
definitions of $X_k$ and $Y_k$ spaces.
}
\end{remark}
%

\end{proposition}


\subsection{Case $\alpha = 0$}

As mentioned, we begin with the Rayleigh operator $Ray_0$ when
$\alpha = 0$. We will find the inverse of $Ray_0$. More precisely, we will construct the Green function of $Ray_0$ and solve  
\begin{equation}\label{Ray0} Ray_0 (\phi) = (U-c) \dz^2 \phi - U'' \phi = f\end{equation}
with boundary condition: $\partial_z \phi_{\vert_{z=1}} = 0$. 
%
We recall that $z_c$ is defined by solving the equation $U(z_c) = c$. We first prove the following lemma.

\begin{lemma}\label{lem-defphi012} Assume that $\I c \not =0$. There are two independent solutions $\phi_{1,0},\phi_{2,0}$ of $Ray_0(\phi) =0$ with the Wronskian determinant 
$$ 
W(\phi_{1,0}, \phi_{2,0}) := \dz \phi_{2,0} \phi_{1,0} - \phi_{2,0} \dz \phi_{1,0} = 1.
$$
Furthermore, there are analytic functions $P_1(z), P_2(z), Q(z)$ with $P_1(z_c), P_2(z_c) , Q(z_c)\not=0$ so that the asymptotic descriptions 
\begin{equation}\label{asy-phi012} 
\phi_{1,0}(z) = (z-z_c) P_1(z) ,\qquad \phi_{2,0}(z) = P_2(z) + Q(z) (z-z_c) \log (z-z_c)
\end{equation}
hold for $z$ near $z_c$. Here when $z-z_c$ is on the negative real axis,
 we take the value of $\log (z-z_c)$ to be $ \log |z-z_c| - i \pi$.   
In particular, $\phi_{1,0}$ is a smooth $C^{\infty}$ function, wherease $\phi_{2,0} \in Y_k$ for every $k \ge 0$.
 \end{lemma}
\begin{proof} First, we observe that $$ \phi_{1,0}(z) = U(z)-c$$ is an exact solution of $Ray_0(\phi) =0$. In addition, the claimed asymptotic description for $\phi_{1,0}$ clearly holds for $z$ near $z_c$ since $U(z_c) = c$. We then construct a second particular solution $\phi_{2,0}$, imposing the Wronskian determinant to be one:
$$
W[\phi_{1,0},\phi_{2,0}] =  \dz \phi_{2,0} \phi_{1,0} - \phi_{2,0} \dz \phi_{1,0}  = 1.
$$
From this, the variation-of-constant method $\phi_{2,0} (z)= \phi_{1,0}(z) u(z)$ then yields
$$
  \phi_{1,0} u \partial_z  \phi_{1,0} 
+ \ \phi_{1,0}^2 \partial_z u -\partial_z  \phi_{1,0} u  \phi_{1,0}  = 1.
$$
This gives $
\partial_z u(z) = 1/ \phi^2_{1,0}(z)$ and therefore
\beq \label{defiphi2}
\phi_{2,0}(z) = (U(z) - c)  \int_{1/2}^z {1 \over (U(y) - c)^2} dy .
\eeq 
Note that $\phi_{2,0}$ is well defined if the denominator does not vanishes, hence if
$\I c \not =  0$ or if $\I c = 0$ and $0 \le z < z_c$. 
More precisely, with denoting $U'_c = U'(z_c)$,
$$
{1 \over (U(z) - U(z_c))^2}
= {1 \over ( U'_c (z - z_c) + U''_c (z - z_c)^2 / 2 + ...)^2}
$$
$$
= {1 \over (U'_c)^2 (z-z_c)^2} 
- {U''_c \over (U'_c)^3} {1 \over z - z_c} + \mbox{ holomorphic.}
$$
Hence
\beq \label{defiphi2bis}
\phi_{2,0} = - {U(z) - c \over (U'_c)^2 (z - z_c)}
- {U''_c \over (U'_c)^3} (U(z) -  c)  \log (z - z_c) + \mbox{ holomorphic.}
\eeq
As $\phi_{2,0}$ is not properly defined for $z < z_c$ when $z_c \in \rit^-$, it is coherent
to choose the determination of the logarithm which is defined on $\cit - \rit^-$.

With such a choice of the logarithm, $\phi_{2,0}$ is holomorphic in $\cit - \{ z_c + \rit^-  \}$.
In particular if $\I z_c = 0$, $\phi_{2,0}$ is holomorphic in $z$ excepted on the half line $z_c + \rit^- $.
For $z \in \rit$, $\phi_{2,0}$ is holomorphic as a function of $c$ excepted if $z - z_c$ is real and positive,
namely excepted if $z < z_c$. 
For a fixed $z$, $\phi_{2,0}$ is an holomorphic function of $c$ provided $z_c$ does not cross
$\rit^-$, and provided $z - z_c$ does not cross $\rit^-$.
The lemma then follows from the explicit expression (\ref{defiphi2bis}) of $\phi_{2,0}$.
\end{proof}

Let $\phi_{1,0},\phi_{2,0}$ be constructed as in Lemma \ref{lem-defphi012}. Then the Green function $G_{R,0}(x,z)$ of the $Ray_0$ operator, taking into account of the boundary conditions, can be defined by 
$$
G_{R,0}(x,z) = \left\{ \begin{array}{rrr} (U(x)-c)^{-1} \phi_{1,0}(z) \phi_{2,0}(x), 
\quad \mbox{if}\quad z>x,\\
(U(x)-c)^{-1} \phi_{1,0}(x) \phi_{2,0}(z), \quad \mbox{if}\quad z<x.\end{array}\right.
$$ 
Here we note that $c$ is complex with $\I c \not=0$ and so the Green function $G_{R,0}(x,z)$ is a well-defined function in $(x,z)$, continuous across $x=z$, and its first derivative has a jump across $x=z$. Let us now introduce the inverse of $Ray_0$ as 
\begin{equation}\label{def-RayS0}
\begin{aligned}
RaySolver_0(f) (z)  &: =  \int_0^1 G_{R,0}(x,z) f(x) dx.
\end{aligned}
\end{equation}
\begin{lemma}\label{lem-defRayS0} For any $f\in {X_0} = Y_0$,  the function $RaySolver_0(f)$ is a solution to the Rayleigh problem \eqref{Ray0}, with $\partial_z RaySolver_0(f)(1) =0$. In addition, $RaySolver_0(f) \in Y_2$, and there holds  
$$
\| RaySolver_0(f)\|_{Y_2} \le C \|f\|_{{X_0}},
$$  for some universal constant. 
\end{lemma}
Note that $Y_k$ spaces are somehow better adapted to Rayleigh equation, since the singularity comes
from $(z - z_c) \log(z - z_c)$ which appears only after taking two derivatives.

\begin{proof} By definition, we have
$$
\begin{aligned}
RaySolver_0(f) (z)  = \phi_{1,0}(z) \int_0^z \phi_{2,0}(x) {f(x) \over U(x) - c} dx
+ \phi_{2,0}(z) \int_z^{1} \phi_{1,0}(x) {f(x) \over U(x) - c} dx .
\end{aligned}
$$
Using the definition of the function space ${X_0}$ and the asymptotic expansion of
 $\phi_{2,0}(z)$ for $z$ near $z_c$
 obtained in Lemma \ref{lem-defphi012}, we have 
$$\begin{aligned}
\Big| \phi_{1,0}(z) \int_0^z \phi_{2,0}(x) {f(x) \over U(x) - c} dx\Big| 
&\le C\|f\|_{X_0} |z-z_c| \int_{0}^z \frac{1}{|x-z_c|} \; dx 
\\& \le C\|f\|_{X_0} |z-z_c| \Big( 1+|\log (z-z_c)|\Big),
\\& \le C\|f\|_{X_0},
\end{aligned}$$
and similarly, 
$$
\Big| \phi_{2,0}(z)  \int_z^1 \phi_{1,0}(x) {f(x) \over U(x) - c} dx \Big| = \Big| \phi_{2,0}(z) \int_z^1 f(x) dx \Big| \le C\|f\|_{X_0} .
$$ 
Hence
\begin{equation}\label{Y0-bound}
\| RaySolver_0(f) \|_{Y_0} \le C \| f \|_{X_0}.
\end{equation}
Next, we write
$$
\begin{aligned}
\dz RaySolver_0(f) (z)  = \dz \phi_{1,0}(z) \int_0^z \phi_{2,0}(x) {f(x) \over U(x) - c} dx
+ \dz \phi_{2,0}(z) \int_z^{1} \phi_{1,0}(x) {f(x) \over U(x) - c} dx .
\end{aligned}
$$
The boundary condition follows directly by the assumption that $\partial_z \phi_{1,0}(1) = U'(1)=0$. Now, near $z = z_c$, $\dz \phi_{2,0}(z)  = \cO(\log (z-z_c))$, and so a similar estimate to those given just above yields 
$$
| \dz RaySolver_0(f) (z) | \le C \|f\|_{X_0} (1 + | \log(z - z_c)| ).
$$
Hence
$$\| RaySolver_0(f) \|_{Y_1} \le C \| f \|_{X_0}.
$$
For the second derivative, we write 
\begin{equation}\label{d2-eqs} \dz^2 (RaySolver_0(f)) = \frac{U''}{U-c} RaySolver_0(f) + \frac{f}{U-c},\end{equation}
which proves at once $\|RaySolver_0(f) \|_{Y_2} \le C \|f\|_{X_0}$. 
\end{proof}

The following lemma is then straightforward and will be of use in the latter sections. 

\begin{lemma}\label{lem-RayS0} Let $k \ge 2$.
For any $f \in X_k$, the function $RaySolver_0(f)$ belongs to $Y_{k+2}$, and there holds 
$$
\| RaySolver_0(f) \|_{Y_{k+2}} \le C \|f\|_{X_k} $$ 
for some universal constants $C_k$.
\end{lemma}
\begin{proof} The lemma follows directly from taking derivatives of the identity \eqref{d2-eqs}, and using the estimates obtained in Lemma \ref{lem-defRayS0}, since each time we derive, we lose an $(U - c)$ factor.  
\end{proof}


\subsection{Case $\alpha \ne 0$: the exact Rayleigh solver}


Let us prove in this section Proposition \ref{prop-exactRayS} 

%
%

\begin{proof}[Proof of Proposition \ref{prop-exactRayS}]
Note that for any function $f$, we have
$$
Ray_\alpha( RaySolver_0(f)) = f - \alpha^2 (U-c) RaySolver_0(f).
$$
We therefore build the Rayleigh solver $RaySolver_\alpha(\cdot)$ by iteration, defining iteratively
$$
S_0(f)  := RaySolver_0(f), \qquad S_j(f) := RaySolver_0 \Bigl( \alpha^2 (U-c) S_{j-1} (f) \Bigr) , 
$$
for any $f \in Y_0$. The exact Rayleigh solver of the Rayleigh equation is then defined by 
\begin{equation}\label{def-RaySa}
RaySolver_\alpha (f): = \sum_{j = 0}^{+\infty} S_j (f), \qquad f \in Y_0.
\end{equation}
Indeed, since $f \in Y_0$, then by the estimate \eqref{Y0-bound} and iteration, $S_j(f) \in Y_0$ and 
$$ \|S_j(f)\|_{Y_0} \le C \alpha^2 \| S_{j-1}(f)\|_{Y_0} \le C^j \alpha^{2 j} \| f\|_{Y_0}.$$
For sufficiently small $\alpha$, the series $\sum_{j=0}^{+ \infty} S_j(f)$ is thus convergent in $Y_0$. In addition, for all $J\ge 0$, 
$$
Ray_\alpha \Bigl( \sum_{j=0}^{J} S_j(f) \Bigr) = f - \alpha^2 (U-c) S_J(f).
$$
By taking $J \to \infty$, $\sum_{j=0}^{+ \infty} S_j(f)$ defines the Rayleigh solver from $Y_0$ to $Y_0$. More generally, if $f \in Y_k$ for some $k \ge 0$, then the function $
RaySolver_\alpha(f) $ lies in $Y_k$.
Proposition \ref{prop-exactRayS} then follows by combining with Lemma \ref{lem-RayS0}.
\end{proof}


\subsection{Case $\alpha \ne 0$: two particular solutions}


\begin{lemma}\label{lem-phi1alpha} For $\alpha$ small enough, 
there exists two functions $\phi_{j,\alpha} \in Y_4$ with $j = 1,2$, uniformly
bounded in $Y_4$ as $\alpha$ goes to $0$, such that
$$
Ray_\alpha ( \phi_{j,\alpha} ) = 0,
$$
$$
\| \phi_{1,\alpha} - \phi_{1,0} \|_{Y_4} = O(\alpha^2).
$$
Morevoer
$$\begin{aligned}
\phi_{1,\alpha}(0) &= U_0 - c + \frac{\alpha^2 }{U'_0}\int_0^1 (U-c)^2 \; dx + \cO(\alpha^2 z_c \log z_c)
\\
\partial_z \phi_{1,\alpha} (0) &= U'_0 + \cO(\alpha^2\log z_c).
\end{aligned}$$
\end{lemma}

\begin{proof}
We use the previous construction to build exacts solution of $Ray_\alpha(\phi) = 0$, starting from
$\phi_{1,0}$ and $\phi_{2,0}$, the two solutions of $Ray_0(\phi) = 0$ that are constructed 
above in Lemma \ref{lem-defphi012}. 
Let us denote 
$$ 
\phi_{j,n} (z): = RaySolver_0\Big( (U-c) \phi_{j,n-1} \Big)(z) ,
$$ 
for $n \ge 1$.  Clearly, we have 
$$Ray_\alpha (\sum_{k=0}^n  \alpha^{2k} \phi_{j,k}) = - \alpha^{2(n+1)} (U-c) \phi_{j,n}.$$
By the estimate obtained in Lemma \ref{lem-defRayS0}, 
we get $\| \phi_{j,n} \|_{Y_0} \le C \|\phi_{j,n-1}\|_{Y_0}$, for all $n\ge 1$.
 Therefore the series in the above equation converges in ${Y_0}$ for sufficiently small $\alpha$. 
 This proves that   
\begin{equation}\label{def-aphi12} 
\phi_{j,\alpha} (z): = \sum_{n=0}^\infty  \alpha^{2n} \phi_{j,n} (z)
\end{equation}
are well-defined in $Y_0$ and are two exact solutions to $Ray_\alpha \phi = 0$. 

We now detail the first terms of the asymptotic expansions of $\phi_{j,\alpha}$.
%
First, we recall that $\phi_{1,0}(0) = U_0 - c$ with $U_0 = U(0)$, 
and $\partial_z \phi_{1,0}(0;\epsilon,c) = U'_0 \not =0$. In addition, since $z_c$ is sufficiently small, we can write 
\begin{equation}\label{exp-phi20}
\phi_{2,0}(0) =  {1 \over {U'_0}}  -  {2U''_0\over {U'_0}^2 } z_c \log z_c + \cO(z_c), \qquad \dz \phi_{2,0}(0) =   {2U''_0\over {U'_0}^2 } \log z_c + \cO(1),
\end{equation}
with $U_0' = U'(0)$ and $U''_0 = U''(0)$. We also recall that $z_c$ is a complex number with $U(z_c) = c$, and so $\I c = U'_0 \I z_c + \cO(z_c^2)$.

Next, by definition, $\phi_{j,1} = RaySolver_0((U-c) \phi_{j,0})$. That is, we have  
$$
\begin{aligned}
\partial_z^k \phi_{j,1}(0;\epsilon,c)  &=  \partial_z^k\phi_{2,0}(0;\epsilon,c) \int_0^{1} \phi_{1,0}(x) \phi_{j,0}\; dx 
\end{aligned}
$$
for $j = 1,2$ and $k = 0,1$. This proves the Lemma. 
 \end{proof}

\newpage 
 
\section{Airy equations}\label{sec-Airy}


Our goal is to inverse the Airy operator defined as in \eqref{opAiry}, and thus we wish to construct the Green function for the modified Airy equation
\beq \label{Airyp}
\eps \partial_z^4 \phi - (U(z)-c + 2 \eps \alpha^2) \dz^2 \phi = 0.
\eeq
In the first paragraph we recall classical properties of the classical Airy equations. In the second one we
detail Langer transformation, and then focus on the construction of the Green function.


\subsection{Airy functions}


The aim of this section is to recall some properties of the classical Airy functions.
 The classical Airy equation is
\beq \label{Airy}
\dz^2 \phi - z \phi = 0, \qquad z \in \CC,
\eeq
with two classical solutions named $Ai(z)$ and $Bi(z)$, which go to $0$ respectively
at $+ \infty$ and $- \infty$.
In connection with the Orr-Somerfeld equation with $\epsilon$ being complex, we are interested in the Airy functions with a {\it complex} argument 
$$
z = e^{i \pi / 6} x, \quad x \in \rit .
$$
We have therefore to introduce two independent solutions which converge to $0$ respectively at
$+ \infty$ and $- \infty$ on this {\it complex} line. We will take $Ai$ and 
$$
Ci = -  i \pi (Ai  + i Bi) .
$$
We will need the following estimates, which proofs may be 
found in \cite{Airy,Vallee}; see also \cite[Appendix]{Reid}.

\begin{lemma}\label{lem-classicalAiry} The classical Airy equation \eqref{Airy}
has two independent solutions $Ai(z)$ and 
$Ci(z)$ so that the Wronskian determinant of $Ai$ and $Ci$ equals
\beq \label{WCiAi}
W(Ai, Ci)  =  Ai(z) Ci'(z) -Ai'(z) Ci(z) =   1 .
\eeq
In addition, $Ai(e^{i \pi /6} x)$ and $Ci(e^{i \pi /6} x)$ converge to $0$ as $x\to \pm \infty$
($x$ being real), respectively. Furthermore, there hold asymptotic bounds: 
\beq \label{boundAik1}
\Bigl| Ai(k, e^{i \pi / 6} x) \Bigr| \le
 {C| x |^{k/2-1/4} } e^{-\sqrt{2 | x|} x / 3}, \qquad k\in \ZZ, \quad x\in \RR,
 \eeq
 and
 \beq \label{boundCik1}
\Bigl| Ci(k, e^{i \pi / 6} x) \Bigr| \le
 {C| x |^{k/2-1/4} } e^{\sqrt{2 | x|} x / 3}, \qquad k\in \ZZ,\quad x\in \RR,
 \eeq
in which $Ai(0,z) = Ai(z)$, $Ai(k,z) = \partial_z^{-k} Ai(z)$ for $k\le 0$, and $Ai(k,z)$ is the $k^{th}$ primitive of $Ai(z)$ for $k\ge 0$ and is defined by the inductive path integrals
$$
Ai(k, z ) = \int_\infty^z Ai(k-1, w) \; dw 
$$
so that the integration path is contained in the sector with $|\arg(z)| < \pi/3$. The Airy functions $Ci(k,z)$ for $k\not =0$ are defined similarly.  
\end{lemma}

Note that $Ai(x)$ is rapidly decreasing as $x$ goes to $+\infty$, and $Ci(x)$ is rapidly increasing as
$x$ goes to $+\infty$. Note also that two derivatives lead to a multiplication by a factor $x$.

\subsection{Green function of Airy equation}


Using the properties of the Airy solutions, we can now construct the Green function for the classical
Airy equation. More precisely, let us consider
\beq \label{AiryG}
\eps \dz^2 \phi -\lambda  (z-z_1)  \phi = f, \qquad z\in [0,1],
\eeq
where $\lambda$ is some positive constant, and $z_1$ is in $[0,1]$. Let
$$
\delta = \Bigl( { \eps \over \lambda} \Bigr)^{1/3} 
=  i^{-1/3}Ê(\lambda\alpha R)^{-1/3} = e^{- i \pi / 6} (\lambda\alpha R)^{-1/3} 
$$
which can be interpreted as the critical layer size. Note that $
\arg(\delta) = - \pi / 6$.
Let $G_a(x,z)$ be the Green function of \eqref{AiryG} defined by
\beq \label{GreenAiry}
G_a(x,z) =  \left\{ \begin{array}{rrr}   \delta  \eps^{-1}  Ai(X)Ci(Z), 
\qquad &\mbox{if}\qquad x>z,\\
 \delta \eps^{-1} Ai (Z) Ci(X) , \qquad &\mbox{if}\qquad x<z,
\end{array}\right.
\eeq
with $X: = \delta^{-1} (x-z_1)$ and $Z: = \delta^{-1} (z-z_1)$. 
Here note that by \eqref{WCiAi}, $[G_a(x,z)]_{\vert_{x=z}} = \lim_{x \to z^+} G_a(x,z) - \lim_{x\to z^-} G_a(x,z) = 0$ and $[\D_z G_a(x,z)]_{\vert_{x=z}} =- [\D_x G_a(x,z)]_{\vert_{x=z}} = \epsilon^{-1}$.  

The solution $\phi$ of (\ref{AiryG}) is given by the convolution
\beq \label{GreenAiry2}
\phi(z) =  \int_{0}^{1} G_a(x,z) f(x) dx .
\eeq
Let us detail some estimates on $G_a$ as a warm up for the following sections.
First $G_a(x,x)$ is uniformly bounded in $x$. From the asymptotic properties of 
$Ai$ and $Ci$ we can in fact get
\begin{equation}\label{Ga-ptw}
| G_a(x,z) | \le  C_0 \delta^{-2}
\end{equation}
for some constant $C_0$, uniformly in $x$ and $z$. Here and in what follows, bounds involving $\delta$ are understood as those in term of $|\delta|$. As a consequence (\ref{GreenAiry2}) is well defined as soon as 
$f$ is integrable and
\begin{equation}\label{roughG0}
\Big\|  \int_{0}^{1} G_a(x,\cdot) f(x) dx  \Big  \|_{L^\infty} \le  C_0\delta^{-2} \| f \|_{L^1} .
\end{equation}
Let us get sharper bounds on the Green function.  Lemma \ref{lem-classicalAiry} yields
$$| \D_z^k Ai(e^{i \pi / 6} z)  \D_x^\ell Ci(e^{i \pi / 6} x) | \le 
{C(1+|z|)^{k/2-1/4} (1+|x|)^{\ell/2-1/4}} 
\exp \Bigl( {1\over 3} \sqrt{2 |x|} x -  {1\over 3} \sqrt{2 |z|} z \Bigr) .
$$
Consider the case: $x<z$.
This in particular yields 
$$
| \D_z^k Ai(e^{i \pi / 6} z)  \D_x^\ell Ci(e^{i \pi / 6} x) | \le 
{C(1+|z|)^{k/2-1/4} (1+|x|)^{\ell/2-1/4}} 
\exp \Bigl(- {1\over 3} \sqrt{2 |z|} |x-z| \Bigr) .$$
A symmetric  bound holds true for the case $x>z$, which leads to
\begin{equation}\label{Ga-boundzx}
| \D_z^k \D_x^\ell G_a(x,z) | \le 
C  \delta^{-k-\ell-2} { (1+|Z|)^{k/2-1/4} (1+|X|)^{\ell/2-1/4}} 
\exp \Bigl( -  {1\over 3} \sqrt{2 |Z|} |X-Z| \Bigr) ,
\end{equation} for all $x$ and $z$, with $Z = (z-z_1)/\delta$ and $X = (x-z_1)/\delta$. 

\subsection{Green function of the primitive Airy equation}\label{sec-PrimAiry}


In the sequel, we will have to look for the Green function of the equation
\beq \label{AiryG2}
\eps \partial_z^4 \phi  - \lambda (z-z_1) \dz^2 \phi = 0.
\eeq
We will therefore have to integrate twice $G_a(x,\cdot)$, and the primitives of $Ai(z)$ and $Ci(z)$ are in use.  We shall choose $G_{2,a}(x,z)$ such that $\D_z^2 G_{2,a}(x,z)$ coincides with the Green function $G_a(x,z)$ of the Airy function  constructed in the previous subsection. 
We construct the Green function $G_{2,a}(x,z)$ as follows:
\begin{equation}\label{Green2} 
G_{2,a}(x,z) = \delta^3 \eps^{-1}  \left\{ \begin{array}{rrr} Ai(X)Ci(2,Z) +  a_1 (X) Z, &\mbox{if }x > z,\\
 Ai(2,Z) Ci(X) + a_2 (X) + a_1(X) X ,  &\mbox{if } x < z,
\end{array} \right.
\end{equation}
with the notation $X = \delta^{-1} (x-z_1) $ and $Z = \delta^{-1} (z-z_1)$,
where $a_1(X)$ and $a_2(X)$ are chosen such that $G_{2,a}(x,z)$
and $\partial_z G_{2,a}(x,z)$ 
are continuous at $x = z$, namely
\beq \label{a1}
a_1(X) :=  Ci(X) Ai(1,X)  - Ai(X) Ci(1,X)   ,
\eeq
and \beq \label{a2}
a_2(X) :=  Ai(X) Ci(2,X ) -  Ci(X) Ai(2, X)  .
\eeq

\begin{lemma} Let $G_{2,a}(x,z)$ be defined as in \eqref{Green2}. Then, $G_{2,a}(x,z)$ is indeed the Green function of \eqref{AiryG2}, that is, there holds
$$\eps \partial_z^4 G_{2,a}(x,z) - \lambda z \dz^2 G_{2,a}(x,z)  = \delta_x(z)$$
for each fixed $x$. In addition, for each fixed $z$, $G_{2,a}(\cdot,z)$ solves the adjoint equation of \eqref{AiryG2}, that is
$$\eps \partial_x^4 G_{2,a}(x,z) - \lambda \partial_x^2 (x G_{2,a}(x,z))  = \delta_z(x).$$  
\end{lemma}
\begin{proof} By construction, $\D_z^2 G_{2,a}(x,z) = G_a(x,z)$, which solves \eqref{AiryG}, and so $G_{2,a}(x,z)$ solves \eqref{AiryG2}. In addition, $\D_z^2 G_{2,a}(x,z)$ is continuous
at $x=z$ and $\D_z^3 G_{2,a}(x,z)$ has a jump across $z=x$ which is equal to
$$ 
[\D_z^3 G_{2,a}(x,z)]_{\vert_{x=z}} = [G_{a}(x,z)]_{\vert_{x=z}}  =   \eps^{-1}.
$$
Next, by direct computations we observe that
$\D_x G_{2,a}(x,z)$ 
 and $\D_x^2 G_{2,a}(x,z)$ are also continuous at $x=z$, and that the jump of $\D_x^3G_{2,a}(x,z)$ across $x=z$ is 
$$ [\D_x^3 G_{2,a}(x,z)]_{\vert_{x=z}} =  \eps^{-1} W[Ci,Ai] =  - \eps^{-1}.$$
Furthermore, direct calculations yield
$$\begin{aligned} 
a_1''(X) &= X a_1(X) -1, 
\\
(a_2(X) - Xa_1(X))'' &= X (a_2(X) - Xa_1(X)) + X  .
\end{aligned}$$
This proves 
$$\begin{aligned}
\epsilon \D_x^2 G_{2,a}(x,z) &= \lambda x G_{2,a}(x,z) + (z-x) \chi_{\{x>z\}} 
\\ \epsilon \D_x^4 G_{2,a}(x,z) &= \lambda x \D_x^2 G_{2,a}(x,z) + 2 \lambda \D_x G_{2,a}(x,z),
\end{aligned}$$ 
where $ \chi_{\{x>z\}}$ equals one if $x>z$ and zero if otherwise. That is, $G_{2,a}(x,z)$ solves the adjoint equation as claimed. 
\end{proof}

For our convenience, we denote 
\begin{equation}\label{Green-decomp} G_{2,a}(x,z) := \widetilde G_{2,a}(x,z) + E_{2,a}(x,z),\end{equation} where $\widetilde G_{2,a}(x,z)$ denotes the localized behavior in the Green function $G_{2,a}(x,z)$, and  $E_{2,a}(x,z)$ denotes the linear term in the Green function, that is 
\begin{equation}\label{def-E2a} 
E_{2,a}(x,z) =  \delta^3 \eps^{-1}   \left\{ \begin{array}{rrr} a_1 (X) Z, &\mbox{if }x > z,\\
  a_2 (X) + a_1(X) X,  &\mbox{if } x < z,
\end{array} \right.
\end{equation}
By a view of the bounds \eqref{boundAik1} and \eqref{boundCik1}, it follows easily that  
\begin{equation}\label{bound-a12}
\begin{aligned}
|\D_X^ka_1(e^{i\pi / 6}X)| \le C(1+|X|)^{(k-2)/2}, \qquad 
|\D_X^k a_2(e^{i\pi/6}X)| \le C(1+|X|)^{(k-3)/2} 
\end{aligned}
\end{equation} for all $X \in \RR$ and integers $k\ge 0$. In addition, there hold pointwise bounds on the Green function:
\begin{equation}\label{G2a-boundzx0}
\begin{aligned}
| \D_z^\ell \D_x^k \widetilde G_{2,a}(x,z) | 
\le 
C \delta^{-(k+\ell)}|X|^{k/2-1/4} |Z|^{\ell/2-5/4} 
e^{\sqrt{2 |X|} X /3 -  \sqrt{2|Z|} Z /3} ,
\end{aligned}
\end{equation} for all $X<Z$; a symmetric bound is valid for the case $X>Z$. 
Note that
\begin{equation}\label{G2a-boundzx}
\begin{aligned}
| \D_z^\ell \D_x^k \widetilde G_{2,a}(x,z) | 
\le 
C \delta^{-(k+\ell)}|X|^{k/2-1/4} |Z|^{\ell/2-5/4} 
e^{- \sqrt{ 2 |Z|} (Z  - X)  /3 } .
\end{aligned}
\end{equation} 
Furthermore, in the case where $X$ and $Z$ are away from each other: $|X|\le \frac 12 |Z|$ or $|X|\ge 2|Z|$, we have,
using (\ref{G2a-boundzx0}),  
\begin{equation}\label{G2a-boundawayzx}
\begin{aligned}
| \D_z^\ell \D_x^k \widetilde G_{2,a}(x,z) | 
\le 
C \delta^{-k-\ell} e^{ - \frac 16 |Z|^{3/2}} e^{ - \frac 16 |X|^{3/2} }.
\end{aligned}
\end{equation} 

%

%

%



 \subsection{Langer transformation} \label{sec-Langer}


Since the profile $U$ depends on $z$ in a non trivial manner, we make a change of variables and unknowns in order to go back
to classical Airy equations studied in the previous section. This change is very classical in physical literature, and
called  the Langer's transformation.

\begin{definition}\label{def-Langer} By Langer's transformation $(z,\phi) \mapsto (\eta,\Phi)$, we mean $\eta = \eta(z)$ defined by
\begin{equation}\label{var-Langer}
\eta (z) = \Big[ \frac 32 \int_{z_c}^z \Big( \frac{U-c}{U'_c}\Big)^{1/2} \; dz \Big]^{2/3}
\end{equation}
and $\Phi = \Phi(\eta)$ defined by the relation
\begin{equation}\label{phi-Langer}
 \dz^2 \phi (z) = \dot z ^{1/2} \Phi(\eta),
 \end{equation}
 in which 
 $$
 \dot z = \frac{d z( \eta)}{d \eta} 
 $$ 
 and $z = z(\eta)$ is the inverse of the map $\eta = \eta(z)$. 
\end{definition}
Direct calculation gives a useful fact; $(U-c)\dot z^2  = U'_c \eta$. 
Next, using that $c = U(z_c)$, one observes that for $z$ near $z_c$, we have 
\begin{equation}\label{est-eta}\begin{aligned}
\eta (z) &=  \Big[ \frac 32 \int_{z_c}^z \Big(z-z_c + \frac{U_c''}{U_c'} (z-z_c)^2 
+ \cO(|z-z_c|^3)\Big)^{1/2} \; dz \Big]^{2/3} \\
&= z-z_c + \frac {1}{10} \frac{U_c''}{U_c'} (z-z_c)^2 + \cO(|z-z_c|^3).\end{aligned}
\end{equation}
In particular, we have \begin{equation}\label{est-edot}\eta'(z) = 1 + \cO(|z-z_c|),\end{equation}
and thus the inverse $ z = z(\eta)$ is locally well-defined and locally increasing near $z=z_c$. In addition, 
$$
\dot z = \frac{1}{\eta'(z)} = 1 + \cO(|z-z_c|).
$$
Next, we note that 
$$
\eta'(z)^2 = \frac{ U-c}{ U_c' \eta(z)},
$$
 which is nonzero away from $z = z_c$. Thus, the inverse of $\eta = \eta(z)$ exists for all $z\in [0,1]$.  

The following lemma links \eqref{Airyp} with the classical Airy equation. 

\begin{lemma}\label{lem-Langer} Let $(z,\phi) \mapsto (\eta, \Phi)$ be the Langer's transformation defined as in Definition \ref{def-Langer}. There holds 
\begin{equation}\label{dz2-cal}
\dz^2 (\dot z^{1/2} \Phi(\eta)) = \dot z^{-3/2} \partial_\eta^2 \Phi(\eta) + \dz^2 \dot z^{1/2} \Phi(\eta) 
\end{equation}
Next, assume that $\Phi(\eta)$ solves 
$$\epsilon   \partial^2_\eta \Phi  - U_c' \eta \Phi   = f(\eta).$$
Then, $\phi = \phi(z)$ solves
$$\eps \partial_z^4 \phi - (U(z)-c) \dz^2 \phi =  \dot z ^{-3/2} f(\eta(z))+ \epsilon  \D_z^2 \dot z^{1/2} \Phi(\eta(z)) $$
\end{lemma}
\begin{proof} Derivatives of the identity $ \dz^2 \phi (z) = \dot z ^{1/2} \Phi(\eta)$
are $$\begin{aligned}
\dz^3 \phi(z) &= \dot z ^{-1/2} \partial_\eta \Phi  + \D_z \dot z^{1/2} \Phi
\end{aligned}$$
and 
\begin{equation}\label{dz4-phiLg}
\begin{aligned}\dz^4 \phi(z) &= \dot z ^{-3/2} \partial^2_\eta \Phi  + [\D_z \dot z^{-1/2} + \dot z^{-1} \D_z \dot z^{1/2}] \D_\eta \Phi + \D_z^2 \dot z^{1/2} \Phi
\\&= \dot z ^{-3/2} \partial^2_\eta \Phi  + \D_z^2 \dot z^{1/2} \Phi
.\end{aligned}
\end{equation}
This proves \eqref{dz2-cal}. Putting these together and using the fact that $(U-c)\dot z^2  = U'_c \eta$, we get 
$$\begin{aligned}
\eps \partial_z^4 \phi - (U(z)-c) \dz^2 \phi  &= \epsilon \dot z ^{-3/2} \partial^2_\eta \Phi  
- (U-c) \dot z^{1/2}  \Phi  + \epsilon  \D_z^2 \dot z^{1/2} \Phi
\\
&= \dot z ^{-3/2} f(\eta)+ \epsilon  \D_z^2 \dot z^{1/2} \Phi.
\end{aligned}$$
The lemma follows. 
\end{proof}

%
%


\subsection{An approximate Green function for the  modified Airy equation}


In this section we will construct an approximate Green function for (\ref{Airyp}). To do
this we fulfill the Langer's transformation and first consider
$$
\eps \partial_\eta^2 \Phi - U_c' \eta \Phi = 0. 
$$
Let us denote 
$$
\delta = \Bigl( { \eps \over U_c'} \Bigr)^{1/3}  = e^{-i \pi / 6} (\alpha R U_c')^{-1/3}.
$$
and introduce the notation $X = \delta^{-1} \xi$ and $Z = \delta^{-1} \eta $. 
The Green function of the above classical Airy equation is simply
$$ 
G_{a}(X,Z) =  \left\{ \begin{array}{rrr}  \delta  \eps^{-1} 
Ai(X)Ci(Z), \qquad &\mbox{if}\qquad \xi > \eta,\\\delta  \eps^{-1} Ai (Z) Ci(X) , \qquad &\mbox{if}\qquad \xi < \eta,
\end{array}\right.
$$
which satisfies the jump conditions across $X=Z$:
$$
[G_a(X.Z)]_{\vert_{X=Z}} = 0 , \qquad [\epsilon \partial_ZG_a(X,Z)]_{\vert_{X=Z}} = 1.
$$
By definition, we have 
\begin{equation}\label{eqs-mGa}
\eps \partial_\eta^2 G_{a}(X,Z)  - U_c' \eta G_{a}(X,Z)  = \delta_\xi(\eta).  
\end{equation}
Next, let us take $\xi = \eta(x)$ and $\eta = \eta(z)$ where $\eta(\cdot)$ is the Langer's transformation and denote $\dot x = 1/\eta '(x)$ and $\dot z = 1/ \eta'(z)$. By a view of \eqref{phi-Langer}, we define the function $G(x,z)$ so that 
\beq \label{def-doubleG}
\dz^2G(x,z) =\dot x ^{3/2} \dot z^{1/2}  G_{a}(\delta^{-1}\eta(x),\delta^{-1}\eta(z)),
\eeq 
in which the factor $\dot x ^{3/2}$ was added simply to normalize the jump of $G(x,z)$. It then follows from Lemma \ref{lem-Langer} together with $\delta_{\eta(x)} (\eta(z)) = \delta_x(z)$ that 
\begin{equation}\label{eqs-2G}\eps \partial_z^4 G(x,z) - (U(z)-c) \dz^2 G(x,z) =  \delta_x(z) + \epsilon  \D_z^2 \dot z^{1/2} \dot z^{-1/2} \dz^2G(x,z) . 
\end{equation}
That is, $G(x,z)$ is indeed an approximate Green function of the modified Airy operator $\epsilon \dz^4 - (U-c)\dz^2$ up to a small error term of order $\epsilon \partial_z^2 G = \cO(\delta)$. It remains to solve \eqref{def-doubleG} for $G(x,z)$, retaining the jump conditions on $G(x,z)$ across $x=z$. 

In view of primitive Airy functions, let us denote
$$ \widetilde Ci(1,z) =\delta^{-1} \int_0^z \dot y^{1/2} Ci(\delta^{-1}\eta(y))\; dy, \qquad \widetilde Ci(2,z) = \delta^{-1}\int_0^z \widetilde Ci(1,y)\; dy$$
and 
$$ \widetilde Ai(1,z) = \delta^{-1}\int_\infty^z \dot y^{1/2} Ai(\delta^{-1}\eta(y))\; dy, \qquad \widetilde Ai(2,z) =\delta^{-1} \int_\infty^z \widetilde Ai(1,y)\; dy.$$
Thus,
 we are led to introduce 
\beq \label{def-GreenAiry2} 
G(x,z) =  \delta^3\epsilon^{-1}  \dot x^{3/2} \left\{ \begin{aligned}  
Ai(\delta^{-1}\eta(x)) \widetilde Ci(2,z) + a_1(x) (z-z_c) /\delta , \quad &\mbox{if }x>z,\\
  Ci(\delta^{-1}\eta(x)) \widetilde Ai(2,z) + a_2(x) + a_1(x) (x-z_c)/\delta, \quad  &\mbox{if } x<z,
\end{aligned} \right.
\eeq
in which $a_1(x), a_2(x)$ are chosen so that the jump conditions in \eqref{def-jumpG0} hold. Clearly, by definition, $G(x,z)$ solves \eqref{def-doubleG}. Next, by view of \eqref{eqs-2G}, we require the following jump conditions on the Green function:
\begin{equation}\label{def-jumpG0}
\begin{aligned}
~[G(x,z)]_{\vert_{x=z}} =  [\partial_z G(x,z)]_{\vert_{x=z}}  = [\partial_z^2 G(x,z)]_{\vert_{x=z}}  =0
\end{aligned}\end{equation}
and 
\begin{equation}\label{def-jumpG1}
\begin{aligned}
~[\epsilon \partial_z^3G(x,z)]_{\vert_{x=z}} =  1.
\end{aligned}\end{equation}
We note that from \eqref{def-doubleG} and the jump conditions on $G_a(X,Z)$ across $X=Z$, the above jump conditions of $\partial_z^2 G$ and $\partial_z^3 G$ follow easily.  In order for the jump conditions on $G(x,z)$ and $\partial_zG(x,z)$, we take 
\begin{equation}\label{def-ta12}
\begin{aligned}   a_1 (x)  &=
 Ci(\delta^{-1}\eta(x)) \widetilde Ai(1,x) - Ai(\delta^{-1}\eta(x)) \widetilde Ci(1,x)  ,\\
a_2(x)  
&=  Ai(\delta^{-1}\eta(x)) \widetilde Ci(2,x)  - Ci(\delta^{-1}\eta(x))\widetilde Ai(2,x)  . 
\end{aligned}
\end{equation}

We obtained the following lemma. 

\begin{lemma}\label{lem-GreenAiry} Let $G(x,z)$ be defined as in \eqref{def-GreenAiry2}. Then $G(x,z)$ is indeed an approximate Green function of (\ref{Airyp}). Precisely, there holds
\beq \label{approxGreen}
\eps \partial_z^4  G (x,z) - (U(z)-c + 2\alpha^2 \epsilon ) \dz^2  G(x,z) = \delta_x(z) + Err_A(x,z)
\eeq
where $Err_A(x,z)$ denotes the error kernel defined by 
\begin{equation}\label{def-ErrA} \begin{aligned}
Err_A(x,z)  &=   \epsilon  \D_z^2 \dot z^{1/2} \dot z^{-1/2} \dz^2G(x,z) - 2 \alpha^2 \epsilon \dz^2 G(x,z) . 
\end{aligned}\end{equation}
\end{lemma}

It appears convenient to denote by $\widetilde G(x,z)$ and $E(x,z)$ the localized and non-localized parts of the Green function, respectively. Precisely, we denote 
$$
\widetilde G(x,z) =  \delta^3  \epsilon^{-1}  \dot x^{3/2} \left\{ \begin{array}{rrr}  Ai(\delta^{-1}\eta(x)) \widetilde Ci(2,z) , &\mbox{if }x>z,\\
  Ci(\delta^{-1}\eta(x)) \widetilde Ai(2,z) ,  &\mbox{if } x<z,
\end{array} \right. 
$$ and 
$$
E(x,z) = \delta^3  \epsilon^{-1}  \dot x^{3/2} \left\{ \begin{array}{rrr}  
 a_1 (x) (z-z_c)/\delta , &\mbox{if }x>z,\\
 a_2(x) + a_1(x) (x-z_c)/\delta,  &\mbox{if } x<z.
\end{array} \right.
$$
Let us give some bounds on the Green function, using the known bounds on $Ai(\cdot)$ and $Ci(\cdot)$. We have the following lemma. 

\begin{lemma} \label{lem-ptGreenbound} Let $G(x,z) = \widetilde G(x,z) + E(x,z)$ be the Green function defined as in \eqref{def-GreenAiry2}, and let $X = \eta(x)/\delta$ and $Z = \eta(z)/\delta$. There hold pointwise estimates
\begin{equation}\label{mG-xnearz}
\begin{aligned}
|\D_z^\ell \D_x^k \widetilde G(x,z) | 
\le 
C \delta^{-k-\ell}(1+|Z|)^{(k+\ell-3)/2}
e^  {-  {\sqrt{2} \over 3} \sqrt{|Z|}|X-Z| }  ,
\end{aligned}
\end{equation} for $\frac 12 |Z|\le |X|\le 2|Z|$, and  
\begin{equation}\label{mG-xawayz}
\begin{aligned}
| \D_z^\ell\D_x^k \widetilde G(x,z) | 
\le 
C \delta^{-k-\ell} e^{ - \frac 14 |Z|^{3/2}} e^{ - \frac 14 |X|^{3/2} },
\end{aligned}
\end{equation} 
for $|X|\le \frac 12 |Z|$ or $|X|\ge 2|Z|$. Similarly, for the non-localized term, we have 
\begin{equation}\label{ak12-bound}\begin{aligned}  
 |\partial_x^k a_1 (x)|  \le C\delta^{-k}(1+|X|)^{k/2-1},
 \qquad  |\partial_x^ka_2(x)| \le    C\delta^{-k}(1+|X|)^{k/2 - 3/2} .
\end{aligned}
\end{equation}
In particular, \begin{equation}\label{non-locE}
|E(x,z) |\le   C \times \left\{ \begin{array}{rrr}  
 |Z|(1+|X|)^{-1}, &\mbox{if }x>z,\\
 1,  &\mbox{if } x<z.
\end{array} \right.
\end{equation}
 \end{lemma}
\begin{proof} 
We recall that for $z$ near $z_c$, we can write $ \dot z (\eta(z) ) = 1 + \cO(|z-z_c|),$ which in particular yields that 
$$
\frac 12 \le \dot z(\eta(z)) \le \frac 32 
$$ 
for $z$ sufficiently near $z_c$. 

We compute 
$$\begin{aligned}
 |\widetilde Ai(1,z) | &\le  \delta^{-1}
 \int_z^1 |\dot y^{1/2}Ai(e^{i\pi/6}Y) |\; dy \le  C\delta^{-1}\int_z^1 (1+|Y|)^{-1/4}  \Bigl|e^{-\sqrt{2|Y|} Y/3} \Bigr| \; dy 
 \\& \le  C \delta^{-1} \int_z^1 (1+|Y|)^{-1/4} e^{-\sqrt{2|Y|} \Re Y/3} \; dy 
 \\& \le  C\int_{\Re Z}^\infty (1+|Y|)^{-1/4} e^{-\sqrt{2|Y|}  y'/3} \; dy' 
 \end{aligned}
 $$
where we make the change of variable $y' = \Re Y =  | \delta |^{-1} \Re \eta(y)$. Now for large
$\Re Z$, this expression is bounded asymptotically by 
$$ 
  \le  C (1+|Z|)^{-3/4} e^{-\sqrt{2|Z|} \Re Z/3}  .
  $$
This remains remains true for bounded $Z$.
Similarly
$$\begin{aligned}
 |\widetilde Ai(2,z) | &\le  \delta^{-1}\int_z^1 |\widetilde Ai(1,y) |\; dy \le  C\delta^{-1}\int_z^1 (1+|Y|)^{-3/4} e^{-\sqrt{2|Y|} \Re Y/3} \; dy 
 \\& \le  C\int_Z^\infty (1+|Y|)^{-3/4} e^{-\sqrt{2|Z|} \Re Y/3} \; dY 
 \\& \le  C(1+|Z|)^{-5/4} e^{-\sqrt{2|Z|} \Re  Z/3}  .
  \end{aligned}
$$
Similarly, we have 
$$ \begin{aligned}
|\widetilde Ci(1,z)| &\le \delta^{-1} \int_0^z |\dot y^{1/2}Ci(e^{i\pi/6}Y)|\; dy
\le C\delta^{-1}\int_0^z  (1+|Y|)^{-1/4} e^{\sqrt{2|Y|} \Re Y/3} \; dy
\\& \le C\int_0^z (1+|Y|)^{-1/4} e^{\sqrt{2|Y|}  \Re Y/3} \; dY
\\& \le C (1+|Z|)^{-3/4} e^{\sqrt{2|Z|} \Re Z/3} 
\end{aligned}$$
and 
$$ \begin{aligned}
|\widetilde Ci(2,z)| &\le \delta^{-1} \int_0^z |\widetilde Ci(1,y)|\; dy
\le C\int_0^z (1+|Y|)^{-3/4} e^{\sqrt{2|Y|} \Re Y/3} \; dY
\\& \le C  (1+|Z|)^{-5/4} e^{\sqrt{2|Z|} \Re Z/3}. 
\end{aligned}$$
These estimates become significant when the critical layer is away from the boundary $z = 0$, that is when $\delta \ll |z_c|$. 

By combining together these bounds and those on $Ai(\cdot)$, $Ci(\cdot)$, the claimed bounds on $\widetilde G(x,z)$ follow similarly to those obtained in \eqref{G2a-boundzx} and \eqref{G2a-boundawayzx}. Derivative bounds are also obtained in the same way. 

Finally, using the above bounds on $\widetilde Ai(k,z)$ and $\widetilde Ci(k,z)$, we get 
$$\begin{aligned}  
 |\partial_x^k a_1 (x)|  \le C\delta^{-k}(1+|X|)^{k/2-1},
 \qquad  |\partial_x^ka_2(x)| \le    C\delta^{-k}(1+|X|)^{k/2 - 3/2} ,
\end{aligned}
$$
upon noting that the exponents in $Ai(\cdot)$ and $Ci(\cdot)$ are cancelled out identically. The boundedness of $E(x,z)$ thus follows easily. 

This completes the proof of the lemma. 
\end{proof}

Similarly, we also obtain the following simple lemma. 
\begin{lemma} \label{lem-ptErrA} Let $Err_A(x,z)$ be the error kernel defined as in \eqref{def-ErrA}, and let $X = \eta(x)/\delta$ and $Z = \eta(z)/\delta$. There hold
\begin{equation}\label{ErrA-xnearz}
\begin{aligned}
| \D_z^k \D_x^\ell Err_A(x,z) |
\le 
C \delta^{1-k-\ell} (1+|Z|)^{(k+\ell -1)/2}
e^  {-  {\sqrt{2} \over 3} \sqrt{|Z|}|X-Z| }  ,
\end{aligned}
\end{equation} for $\frac 12 |Z|\le |X|\le 2|Z|$, and  
\begin{equation}\label{ErrA-xawayz}
\begin{aligned}
| \D_z^k \D_x^\ell Err_A(x,z) |
\le 
C \delta^{1-k-\ell} e^{ - \frac 14 |Z|^{3/2}} e^{ - \frac 14 |X|^{3/2} },
\end{aligned}
\end{equation} 
for $|X|\le \frac 12 |Z|$ or $|X|\ge 2|Z|$. 
\end{lemma} 
\begin{proof} We recall that 
$$
Err_A(x,z)=  \epsilon  \D_z^2 \dot z^{1/2} \dot z^{-1/2} \dz^2G(x,z) - 2 \alpha^2 \epsilon \dz^2 G(x,z). 
$$
Thus, the lemma follows directly from
\begin{equation}\label{Ga-boundzx}
| \D_z^k \D_x^\ell G_a(x,z) | \le 
C  \delta^{-k-\ell-2} { (1+|Z|)^{k/2-1/4} (1+|X|)^{\ell/2-1/4}} 
\exp \Bigl( -  {1\over 3} \sqrt{2 |Z|} |X-Z| \Bigr) ,
\end{equation} for all $x$ and $z$, with $Z = (z-z_1)/\delta$ and $X = (x-z_1)/\delta$, and the fact that $G_a \approx \dz^2 G$.  
\end{proof}


\subsection{Convolution estimates}


In this section, we establish the following convolution estimates.  
\begin{lemma}  \label{lem-ConvAiry} Let $G(x,z)$ be the approximate Green function of the modified Airy equation constructed as in Lemma \ref{lem-GreenAiry}.  
Then there is some constant $C$ so that 
\begin{equation}\label{conv-loc}\begin{aligned}
\int_{0}^{1} | \widetilde G(x,z)|   dx &\le  C \delta (1+|Z|)^{-2} ,
\end{aligned}\end{equation}
and \begin{equation}\label{conv-nonloc}\begin{aligned}
\int_{0}^{1}|E(x,z)|   dx &\le C ,
\end{aligned}\end{equation} for all $z\in [0,1]$.
\end{lemma} 

%
\begin{proof} Using the pointwise bounds obtained in Lemma \ref{lem-ptGreenbound}, we have 
$$\begin{aligned}
\int_{0}^{1} |  \widetilde G(x,z)|   dx &\le C_0 \int_0^1 \Big[ (1+|Z|)^{-3/2}e^{-  {\sqrt{2} \over 3} \sqrt{|Z|}|X-Z| }  + e^{ - \frac 14 |Z|^{3/2}} e^{ - \frac 14 |X|^{3/2} } \Big] \; dx \\&\le C \delta (1+|Z|)^{-2} ,
\end{aligned}$$
upon noting that $dx = \delta \dot z^{-1}(\eta (x)) dX$ with $\dot z(\eta(x)) \approx 1$. As for $E(x,z)$, we have 
$$\int_0^1 |E(x,z)|\; dx  =  \int_0^z |E(x,z)|\; dx + \int_z^1 |E(x,z)|\; dx ,$$
in which the first integral is bounded thanks to the boundedness of $E(x,z)$ for $x<z$; see \eqref{non-locE}. Whereas for the second integral, we note that $E(x,z)$ is also bounded when $x>z$ and $|X|\gtrsim |Z|$, and thus the integration over this region is also bounded. In particular, when $z\ge z_c$, then if $x\ge z$, we have $|X|\ge |Z|$. It thus remains to estimate the case when $z\le z_c$. We have$$\int_0^1 |E(x,z)|\; dx \le C + C|Z| \int_z^{z_c} (1+|X|)^{-1}\; dx \le C \Big(1+ \delta |Z| \log (1+|Z|)
\Big),$$
which is bounded by a constant. This proves the lemma. 
\end{proof}


Similarly, we also obtain the following convolution estimate for the error kernel $Err_A(x,z)$.

\begin{lemma}  \label{lem-ErrAiry} Let $Err_A(x,z)$ be the error kernel of the modified Airy equation defined as in Lemma \ref{lem-GreenAiry}.
Then there is some constant $C$ so that 
\begin{equation}\label{conv-ErrAiry}\begin{aligned}
\int_{0}^1 | Err_A(x,z)|   dx &\le  C \delta^2 (1+|Z|) ^{-1}
\end{aligned}\end{equation}
for all $z\in [0,1]$. 
\end{lemma} 
\begin{proof}
Using the pointwise bounds obtained in Lemma \ref{lem-ptErrA}, we have 
$$\begin{aligned}
\int_{0}^{1} | Err_A(x,z)|   dx &\le C_0\delta  \int_0^1 \Big[ (1+|Z|)^{-1/2}e^{-  {\sqrt{2} \over 3} \sqrt{|Z|}|X-Z| }  + e^{ - \frac 14 |Z|^{3/2}} e^{ - \frac 14 |X|^{3/2} } \Big] \; dx \\&\le C \delta ^2(1+|Z|)^{-1} ,
\end{aligned}$$
in which again the extra factor of $\delta$ was due to the change of variable $X = \eta(x)/\delta$. 
\end{proof}

Finally, when $f$ is very localized, we obtain a better convolution estimate as follows. 

\begin{lemma}  \label{lem-locConvAiry} Let $G(x,z)$ and $Err_A(x,z)$ be the approximate Green function and the error kernel of the modified Airy equation constructed as in Lemma \ref{lem-GreenAiry}, and let $f = f(X)$ satisfy $|f(X)|\le  C_f (1+|X|)^{k}e^{-\sqrt 2 |X|^{3/2}/3} $, $k \in \ZZ$. 
Then there is some constant $C$ so that 
\begin{equation}\label{conv-locsource}\begin{aligned}
\int_{0}^1 | \widetilde G(x,z)  f(X)|   dx &\le  C C_f \delta  (1+|Z|)^{k-2}e^{-\sqrt 2 |Z|^{3/2}/3},
\\
\int_{0}^1 | E(x,z)  f(X)|   dx &\le  C C_f \delta .
\end{aligned}\end{equation}
In addition, we also have
\begin{equation}\label{ErrA-locsource}\begin{aligned}
\int_{0}^1 | Err_A(x,z)  f(X)|   dx &\le  C C_f \delta^2 (1+|Z|)^{k -1} e^{-\sqrt 2 |Z|^{3/2}/3},
\end{aligned}\end{equation}
for all $z\in [0,1]$. 
\end{lemma} 
\begin{proof} The proof is straightforward, following those from Lemmas \ref{lem-ConvAiry} and \ref{lem-ErrAiry}. For instance, we have 
$$\begin{aligned}
\int_{0}^{1} &|  \widetilde G(x,z) f(X)|   dx 
\\&\le C_0C_f \int_0^1 \Big[ (1+|Z|)^{k-3/2} e^{-  {\sqrt{2} \over 3} \sqrt{|Z|}|X-Z| }  + (1+|X|)^{k} e^{ - \frac 14 |Z|^{3/2}} e^{ - \frac 14 |X|^{3/2} } \Big] e^{-\sqrt 2 |X|^{3/2}/3} \; dx
\\&\le C C_f \delta  (1+|Z|)^{k-2}e^{-\sqrt 2 |Z|^{3/2}/3}.
\end{aligned}$$
Here, note that the first integration was taken over the region where $|X|\approx |Z|$. The estimates for $E(x,z)$ and $Err_A$ follow similarly. \end{proof}


\subsection{Resolution of modified Airy equation}


In this section, we shall introduce the approximate inverse of the $Airy$ operator. We recall that  $Airy(\phi) = \eps \dz^4 \phi - (U - c) \dz^2 \phi $. Let us study the inhomogeneous Airy equation
\begin{equation}\label{Airyp-in}Airy(\phi)  =  f(z),\end{equation}
for some source $f(z)$. We introduce the approximate solution to this equation by defining 
\begin{equation}\label{eqs-Airyp-S}
AirySolver(f) := \int_{0}^{1} G(x,z) f(x)  dx .
\end{equation}
Then, since the Green function $G(x,z)$ does not solve exactly the modified Airy equation (see \eqref{approxGreen}), the solution $AirySolver(f)$ does not solve it exactly either. However, there holds
\begin{equation}\label{eqs-AirySolver}
Airy(AirySolver(f)) = f + AiryErr(f)
\end{equation}
where the error operator $AiryErr(\cdot)$ is defined by
$$
AiryErr(f) : = \int_{0}^{1} Err_A (x,z) f(x) dx  ,
$$
in which $Err_A (x,z)$ is the error kernel of the Airy operator, defined as in Lemma \ref{lem-GreenAiry}. In particular, from Lemma \ref{lem-ErrAiry}, we have the estimate 
\begin{equation}\label{AiryErr-iter} \| AiryErr(f)  \|_{X_0}  \le C\delta^2 \|f \|_{X_0},\end{equation}
for all $f\in X_0$. That is, $AiryErr(f)$ is indeed of order $\cO(\delta^2)$ in $X_0$. 

For the above mentioned reason, we may now define by iteration an exact solver for the modified Airy operator. Let us start with a fixed $f \in X_0$. Let us define
\begin{equation}\label{iter-Aphi}
\begin{aligned}
\phi_n &= - AirySolver(E_{n-1})
\\
E_n &=  - AiryErr(E_{n-1}) \end{aligned}
\end{equation}
for all $n \ge 1$, with $E_0 = f$. Let us also denote 
$$
S_n = \sum_{k=1}^n \phi_k .
$$
It  follows by induction that 
$$ Airy (S_n) = f + E_n,$$
for all $n\ge 1$. Now by \eqref{AiryErr-iter}, we have 
$$ \| E_n\|_{X_0} \le C \delta^2 \|E_{n-1}\|_{X_0} \le (C\delta^2)^n \| f\|_{X_0}.$$ 
This proves that $E_n \to 0$ in $X_0$ as $n \to \infty$ since $\delta$ is small. In addition, by a view of Lemma \ref{lem-ConvAiry}, we have 
$$ | \phi_n(z) | \le C  \| E_{n-1}\|_{X_0} \le C (C\delta^2)^{n-1}  \|f\|_{X_0}.$$
This shows that $\phi_n$ converges to zero in $X_0$, and furthermore the series 
$$ S_n \to S_\infty$$
 in $X_0$ as $n \to \infty$, for some $S_\infty \in X_0$. We then denote $AirySolver_\infty(f) = S_\infty$, for each $f \in X_0$. In addition, we have $ Airy (S_\infty) = f,$ that is, $AirySolver_\infty(f) $ is the exact solver for the modified Airy operator. 

To summarize, we have proved the following proposition. 
\begin{prop}\label{prop-exactAiry} Assume that $\delta$ is sufficiently small. There exists an exact solver $AirySolver_\infty(\cdot)$ as a well-defined operator from $X_0$ to $X_0$ so that 
$$ Airy(AirySolver_\infty (f)) = f.$$
In addition, there holds 
$$
\| AirySolver_\infty (f) \|_{X_0} \le C \| f \|_{X_0} ,
$$
for some positive constant $C$. 
\end{prop}

In addition, when $f$ is very localized, the bound on $AirySolver_\infty (f)$ is sharper and in particular is of order $\cO(\delta)$. 

\begin{prop}\label{prop-exactAiry-loc} Assume that $\delta$ is sufficiently small and $f = f(X)$ satisfies $|f(X)|\le  C_f (1+|X|)^{k}e^{-\sqrt 2 |X|^{3/2}/3} $, $k \in \ZZ$. The exact solver $AirySolver_\infty(f)$ exists and satisfies the uniform bound
$$
| AirySolver_\infty (f) (z)| \le C C_f \delta ,
$$
for some positive constant $C$. 
\end{prop}

\begin{proof} This is a direct consequence of Proposition \ref{prop-exactAiry}, using the sharper bounds on the convolutions obtain in Lemma \ref{lem-locConvAiry}. 
\end{proof}


 \newpage

\section{Singularities and Airy equations}


In this section, we study the smoothing effect of the modified Airy function. Precisely, let us consider the Airy equation with a singular source:
\begin{equation}\label{Airyp-singular}
Airy(\phi) = \eps \dz^4 \phi - (U-c) \dz^2 \phi =\epsilon  \D_z^4 f(z),
\end{equation}
 with $f \in Y_4$.
 We also assume that 
 \begin{equation}\label{BCs-f} f'(1) = 0.
 \end{equation}
We prove the following:  
 \begin{proposition}\label{prop-boundAiryS}  
 Assume that $z_c,\delta, \alpha \ll 1$, and $f$ satisfies the above assumptions. 
 Then, the $AirySolver(\cdot)$ and $AiryErr(\cdot)$ operators satisfy 
\begin{equation} 
\Big\| AirySolver (\epsilon \D_x^4  f) \Big\|_{X_2} \le   C \|f\|_{Y_4}\delta (1+|\log \delta|)(1+|z_c/\delta|)^{3/2}
\end{equation}
and \begin{equation} 
\Big\| AiryErr (\epsilon \D_x^4 f )\Big\|_{X_2} \le  C \|f\|_{Y_4} \delta^2 (1+|\log \delta|),
\end{equation} 
for some universal constant $C$. 
\end{proposition}

The above proposition follows directly from the two following lemmas.

%

\begin{lemma}\label{lem-mAiry} Assume that $z_c,\alpha, \delta \ll 1$.
 Let $G(x,z)$ be the approximated Green function 
 to the modified Airy equation constructed 
 as in Lemma \ref{lem-GreenAiry} and let $f(z) \in Y_4$.
 There holds a convolution estimate: 
\begin{equation}\label{mphi-bound}
\begin{aligned}
\Big|(z-z_c)^k \dz^k \int_{0}^1 G(x,z) \epsilon \D_x^4 f(x) dx \Big|
 \le    C \|f\|_{Y_4} \delta (1+|\log \delta|)(1+|z_c/\delta|)^{3/2}
\end{aligned}\end{equation}
for all $z\in (0,1)$, and for $k = 0,1,2$. 
\end{lemma}


Similarly, we also have the following. 
\begin{lemma}\label{lem-mErrAiry} Assume that $z_c,\alpha,\delta \ll 1$. 
Let $Err_A(x,z)$ be the error defined as in Lemma \ref{lem-GreenAiry} 
and let $f(z) \in Y_4$. There holds the convolution estimate for $Err_A(x,z)$
\begin{equation}\label{mErrA-bound}
\begin{aligned}
\Big|(z-z_c)^k \dz^k \int_{0}^1 Err_A (x,z) \epsilon\D_x^4 f(x) dx \Big|
 \le     C \|f\|_{Y_4} \delta^2 (1+|\log \delta|) 
\end{aligned}\end{equation}
for all $z\in (0,1)$, and for $k=0,1,2$. 
\end{lemma}


\begin{proof}[Proof of Lemma \ref{lem-mAiry} with $k=0$] By scaling, let us assume that $\| f\|_{Y_4} = 1$. 
To begin our estimates, let us recall the decomposition of $G(x,z)$ into the localized and non-localized part as 
$$G(x,z) = \widetilde G(x,z) + E(x,z),$$
where $\widetilde G(x,z) $ and $E(x,z) $ satisfy the pointwise bounds in Lemma \ref{lem-ptGreenbound}. 
In addition, we recall that $\epsilon \D_x^j G_{2,a}(X,Z)$ and so $\epsilon \D_x^j G(x,z)$ are continuous across $x=z$ for $j=0,1,2$. Using the continuity, we can integrate by parts to get  
\begin{equation}\label{mphi-integral}\begin{aligned}
\phi (z)&=  - \epsilon \int_{0}^1\D_x^3  (\widetilde G + E) (x,z) \D_x f(x) \; dx 
+ \mathcal{B}(z)\\
&= I_\ell (z) + I_e(z) +  \mathcal{B}(z)
\end{aligned}
\end{equation}
Here, $I_\ell(z) $ and $I_e(z)$ denote the corresponding integral that involves $\widetilde G(x,z)$ and $E(x,z)$
 respectively, and $\mathcal{B}_0(z)$ is introduced to collect the boundary terms at $x=0$ and is defined by 
\begin{equation}\label{def-mBdry}
\mathcal{B}(z): = - \epsilon \sum_{k=0}^2 (-1)^k \D_x^k G(x,z) \D^{3-k}_x(f(x))\Big\vert_{x=0}^{x=1} .
\end{equation}


\bigskip
\noindent
{\bf Estimate for the integral $I_\ell (z)$.}  Using the bounds \eqref{mG-xnearz} and \eqref{mG-xawayz} 
on the localized part of the Green function, we can give bounds on the integral term $I_\ell $
 in \eqref{mphi-integral}. 
 
 Consider the case $|z-z_c|\le \delta$. 
 By splitting the integral into two cases according to the estimates \eqref{mG-xnearz} and \eqref{mG-xawayz}, 
 we get
$$\begin{aligned}
|I_\ell (z)| &=\Big| \epsilon \int_{0}^1 \D_x^3  \widetilde G(x,z) \D_x f(x) \; dx \Big| 
\\& \le  \epsilon \int_{\{|x-z_c|\le 3 \delta\}} |\D_x^3  \widetilde G(x,z) \D_x f(x)| \; dx 
+  \epsilon \int_{\{|x-z_c|\ge 3 \delta\}} |\D_x^3  \widetilde G(x,z) \D_x f(x)| \; dx,
 \end{aligned}$$ 
in which $\epsilon \D_x^3 \widetilde G(x,z)$ is bounded and so the first integral on the right is bounded by 
 $$ C \int_{\{|x-z_c|\le  3 \delta\}} | \D_x f(x)| \; dx
   \le C \int_{\{|x-z_c|\le 3 \delta\}} (1+|\log (x-z_c)|) \; dx \le C \delta (1+|\log \delta|) .$$ 
Let us turn to the second integral on the right. Note that 
$$
\delta^{-1} | x - z_c | \ge 3 \ge 3 \delta^{-1} | z - z_c |,
$$
hence, as $X \sim \delta^{-1} (x - z_c)$ and similarly for $z$, $| X | \ge 2 |Z|$.  
We therefore use \eqref{mG-xawayz} for $x$ away from $z$ to get 
$$\begin{aligned}
\epsilon \int_{\{|x-z_c|\ge 3 \delta\}}  |\D_x^3  \widetilde G(x,z) \D_x f(x)| \; dx 
& \le C  \int_{\{|x-z_c|\ge 3 \delta\}}  e^{ - \frac 16 |X|^{3/2} } (1+|\log (x-z_c)|)\; dx 
\\& \le C  \delta (1+|\log\delta|)\int_{\RR} e^{-\frac 16 |X|^{3/2}}  \; dX 
\\& \le C  \delta (1+|\log\delta|).
\end{aligned}$$

Let us now consider the case $|z-z_c|\ge \delta$. We again split the integral in $x$ into two parts $|x-z_c|\le \delta $ and $|x-z_c|\ge \delta$. 
Using that $\epsilon \dx^3 \widetilde  G$ is bounded we get
$$\begin{aligned}
\epsilon \int_{\{|x-z_c|\le \delta\}} |\D_x^3  \widetilde G(x,z) \D_x f(x)| \; dx & \le C  
\int_{\{|x-z_c|\le \delta\}}  (1+|\log (x-z_c)|)\; dx 
\\&  \le C  
 \delta (1+|\log \delta|).
\end{aligned}$$
Next, for the integral over $\{ |x-z_c|\ge \delta\}$, we note that for $|x - z_c | \ge \delta$,
$| \partial_x f(x) | \le C(1 + | \log \delta |)$.  We then
use  the bounds \eqref{mG-xnearz} and \eqref{mG-xawayz} to get 
$$\begin{aligned}
\epsilon &\int_{\{|x-z_c|\ge \delta\}}  |\D_x^3  \widetilde G(x,z) \D_x f(x)| \; dx  
\\&\le C   (1+|\log \delta|) \Big[ \int_{\frac 12 |z| \le |x|\le |2|z| } e^  {- \sqrt{2|Z|}|X-Z| /3} \; dx
 + e^{ - \frac 14 |Z|^{3/2}}  \int_0^1 e^{ - \frac 14 |X|^{3/2} }  \; dx\Big]
\\&\le C(1+|\log \delta|) \delta .
 \end{aligned}$$

 Therefore in all cases,  we have $|I_\ell (z)|\le C\delta (1+|\log \delta|)$ as claimed.

 \bigskip
\noindent
{\bf Estimate for $I_{e}$.} Let us next consider the integral 
 $$\begin{aligned}  
I_{e} (z) =  -  \int_0^1 \epsilon\D_x^3(E(x,z)) \D_x f(x) \; dx
= - \Big[ \int_{\{|x-z_c|\le \delta\}} + \int_{\{|x-z_c|\ge \delta\}} \Big] \epsilon\D_x^3(E(x,z)) \D_x f(x) \; dx,
\end{aligned}$$
with recalling that 
$$
E(x,z) =  \delta^3  \epsilon^{-1}  \dot x^{3/2} \left\{ \begin{array}{rrr}  
 a_1 (x) (z-z_c)/\delta , &\mbox{if }x>z,\\
 a_2(x) + a_1(x) (x-z_c)/\delta,  &\mbox{if } x<z.
\end{array} \right.
$$
Here $a_1(x),a_2(x)$ satisfy the bound 
$$\begin{aligned}  
 |\partial_x^k a_1 (x)|  \le C\delta^{-k}(1+|X|)^{k/2-1},
 \qquad  |\partial_x^ka_2(x)| \le    C\delta^{-k}(1+|X|)^{k/2-3/2},
\end{aligned}
$$
with $X = \eta(x)/\delta \approx (x-z_c)/\delta$. 

In particular, for $|x-z_c|\le \delta$, the integrand $\epsilon\D_x^3(E(x,z)) \D_x f(x)$ is bounded by $C(1+|\log (x-z_c)|)$ for $x<z$ and by $C(1+|\log (x-z_c)|)|z-z_c|/\delta$ for $z<x$.  In the latter case, we have $|z-z_c|\le |z_c|+\delta$ since $z<x$ and $|x-z_c|\le \delta$. Putting these together, we have 
$$\begin{aligned}
\Big|\int_{\{|x-z_c|\le \delta\}} \epsilon\D_x^3(E(x,z)) \D_x f(x) \; dx\Big| &\le C (1+|z_c|/\delta)\int_{\{|x-z_c|\le \delta\}} (1+|\log (x-z_c)|) \; dx 
\\
&\le C \delta (1+|z_c/\delta| )(1+|\log \delta |) 
.
\end{aligned}$$
Next, we consider the integral over $\{|x-z_c|\ge \delta\}$. In this case, since $X\to \infty$ as $\delta \to 0$, $\partial_x^k a_j(x)$ is very large and therefore we have to take several integration by parts to avoid this growth. 
Note that if $| z_c | \le \delta$ the two smaller boundary terms are not present.
$$\begin{aligned}
 \int_{\{|x-z_c|\ge \delta\}} \epsilon\D_x^3E(x,z) \D_x f(x) \; dx 
 &= - \int_{\{|x-z_c|\ge \delta\}} \epsilon\D_x^2E(x,z) \D^2_x f(x) \; dx 
 \\
 &\quad + B_0(z) + B_1(z) + B_3(z) + B_4(z),
 \end{aligned}$$
 where $B_j(z)$ denotes the boundary terms at $x=0, x=1, x=z,$ and at points which satisfy $|x-z_c|=\delta$, respectively. We have $B_1(z) = 0$ since $\dx f(1) = 0$, whereas 
 $$\begin{aligned}
 B_0(z)  = -  \epsilon\D_x^2E(x,z) \D_x f(x) _{\vert_{x=0}}  = - \epsilon \dx^2 \Big[  a_2(x) + a_1(x) (x-z_c)/\delta\Big]  \D_x f(x) _{\vert_{x=0}} .
 \end{aligned}$$
From the bound on $\partial^k_xa_j(x)$ and the fact that at $x=0$, $X \approx z_c/\delta$, we have  
$$|B_0(z)| \le C \delta (1+|\log \delta |) (1+|z_c/\delta|)^{3/2} .$$
Similarly, we have 
$$\begin{aligned}
 B_3(z) 
 &= \epsilon[\D_x^2E(x,z)]_{\vert_{x=z}} \D_x f(z) 
 = \epsilon \Big[ \partial_x^2 a_2(z) + 2\delta^{-1} \partial_x a_1(z)\Big] \D_z f(z),
 \end{aligned}$$ 
which satisfies the bound
$$|B_3(z)| \le C \delta (1+|\log (z-z_c)|) (1+|Z|)^{-1/2} \le C \delta (1+|\log \delta|).$$
The last boundary term reads 
$$\begin{aligned}
 B_4(z) 
 &= \epsilon\D_x^2E(x,z) \partial_x f(x) {\Big\vert_{\{|x-z_z| = \delta\}}}
 \end{aligned}$$ 
This is the same as in the previous case $|x-z_c|\le\delta$: $B_4(z)$ is bounded by $C\delta (1+|\log \delta |)$ for $x<z$ and by $C\delta(1+|\log \delta|)|z_c|/\delta$ for $z<x$.

To summarize, we have so far shown 
$$\begin{aligned}
& \int_{\{|x-z_c|\ge \delta\}} \epsilon\D_x^3E(x,z) \D_x f(x) \; dx 
 \\&= - \int_{\{|x-z_c|\ge \delta\}} \epsilon\D_x^2E(x,z) \D^2_x f(x) \; dx
 + \cO(\delta(1+|\log \delta|)(1+|z_c/\delta|)^{3/2}).
 \end{aligned}$$
As for the integral term on the right, $\epsilon\D_x^2E(x,z)$ remains large for $x>z$ due to the appearance of $|z-z_c|/\delta$. We again integrate by parts this term to move $x$-derivatives of $E(x,z)$ into $f(x)$. This will leave several boundary terms that are similar to the above and are of a smaller order since the order of derivatives that hit $E(x,z)$ decreases. Thus, we get 
$$\begin{aligned}
& \int_{\{|x-z_c|\ge \delta\}} \epsilon\D_x^3E(x,z) \D_x f(x) \; dx 
 \\&= -\int_{\{|x-z_c|\ge \delta\}} \epsilon E(x,z) \D^4_x f(x) \; dx
 + \cO(\delta(1+|\log \delta|)(1+|z_c/\delta|)^{3/2}).
 \end{aligned}$$
By definition of $E(x,z)$ and the assumption on $f(x)$, we have 
$$\begin{aligned}
\Big | \int_{\{|x-z_c|\ge \delta\}} \epsilon E(x,z) \D^4_x f(x) \; dx\Big | 
&\le C\delta^3\int_{\{|x-z_c|\ge \delta\}}\Big[ 1+ |x-z_c|^{-3}
 \Big]\; dx
\\
&\le C \delta 
\end{aligned}$$
To summarize, we have obtained
$$ |I_e(z)| \le C\delta(1+|\log \delta|)(1+|z_c/\delta|)^{3/2}, $$
for all $z\in (0,1)$.

 \bigskip
 
 \noindent
 {\bf Estimate for the boundary term $\mathcal{B}(z)$.} It remains to give estimates on
$$
\mathcal{B}(z): = - \epsilon \sum_{k=0}^2 (-1)^k \D_x^k G(x,z) \D^{3-k}_x(f(x))\Big\vert_{x=0}^{x=1} .
$$
We claim that 
$$ |\mathcal{B}(z)| \le C\delta (1+|\log \delta|)(1+|z_c/\delta|)^{3/2}.$$
In the estimate for $I_e(z)$, we have provided estimates for the boundary terms at $x=0,1$ that involve $E(x,z)$. It thus remains to consider the terms involving the localized part $\widetilde G(x,z)$ of the Green function. Using the bounds \eqref{mG-xnearz} and \eqref{mG-xawayz} for $x = 0$, we get   
$$ \begin{aligned}
|\epsilon \D_x^k \widetilde G(x,z) \D^{3-k}_x(f(x))\vert{_{x=0}}  &\le C   \delta^{3-k}  (1+ |z_c|^{-2+k} (1+|\log z_c|))  e^{-\frac 23 |Z|^{3/2}}
\end{aligned}$$
which is bounded by $C \delta (1+|\log\delta|) (1+|z_c/\delta|)^{k-2}$, with $k = 0,1,2$.  At $x=1$, $\D^{k}_x(f(x))$ is bounded for $k = 0,...4$, and $\epsilon \D_x^k \widetilde G(x,z)$ is bounded by $C\delta^{3-k}$, for $k\le 3$, which follows directly from the bounds \eqref{mG-xnearz} and \eqref{mG-xawayz}.  This proves the claimed estimate for $\mathcal{B}(z)$.

\end{proof}

\begin{proof}[Proof of Lemma \ref{lem-mAiry} with $k>0$] We now prove the lemma for the case $k=2$; the case $k=1$ follows similarly. We consider the integral 
$$ \epsilon \int_{0}^1 (z-z_c)^2 \dz^2 \widetilde G(x,z) \D_x^4 f(x) dx  = I_1(z) + I_2(z),$$ with $I_1(z)$ and $I_2(z)$ denoting the 
integration over $\{ |x-z_c|\le \delta\}$ and $\{|x-z_c|\ge \delta\}$, respectively.  
Note that $(U(z)-c)\dot z^{2} =  U'(z_c) \eta(z)$ and recall that $Z = \eta(z)/\delta$ by definition. For the second integral $I_2(z)$, by using the bounds on the Green function for $x$ away from $z$ and for $x$ near $z$, it follows easily that 
$$\begin{aligned} |I_2(z)| &\le C \Big[ \delta \int_{\{|x-z_c|\ge \delta\}}  (1+|Z|)^{1/2} e^{-\frac 23 \sqrt{|Z|} |X-Z|} (1+|x-z_c|^{-1})\; dx
\\&\quad + \epsilon e^{-\frac 14 |Z|^{3/2}} \int_{\{|x-z_c|\ge \delta\}}e^{-\frac 14 |X|^{3/2}} (1+|x-z_c|^{-3}) \; dx   \Big] .\end{aligned}$$
Using $|x-z_c|\ge \delta$ in these integrals and making a change of variable $X = \eta(x)/\delta$ to gain an extra factor of $\delta$, we obtain 
$$\begin{aligned} |I_2(z)| &\le C \delta  \Big[ \int_{\RR}  (1+|Z|)^{1/2} e^{-\frac 23 \sqrt{|Z|} |X-Z|} \; dX 
+ e^{-\frac 18 |Z|^{3/2}} \int_{\RR}e^{-\frac 14 |X|^{3/2}}\; dX  \Big] ,\end{aligned}$$
which is clearly bounded by $C \delta$. It remains to give the estimate on $I_1(z)$ over the region: $|x-z_c|\le \delta$. In this case, we take integration by parts three times. Leaving the boundary terms untreated for a moment, let us consider the integral term
$$\epsilon \int_{\{|x-z_c|\le \delta\}} (z-z_c)^2 \dz^2 \dx^3 \widetilde G(x,z) \D_x f(x) dx.$$
We note that the twice $z$-derivative causes a large factor $\delta^{-2}$ which combines with $(z-z_c)^2$ to give a term of order $|Z|^2$. Similarly, the small factor of $\epsilon$ cancels out with $\delta^{-3}$ that comes from the third $x$-derivative. The integral is therefore easily bounded by 
$$\begin{aligned}  C \Big[ & \int_{\{|x-z_c|\le \delta\}}  e^{-\frac 23 \sqrt{|Z|} |X-Z|} (1+|\log (x-z_c)|)\; dx
\\&\quad + e^{-\frac 14 |Z|^{3/2}} \int_{\{|x-z_c|\le \delta\}}e^{-\frac 14 |X|^{3/2}} (1+|\log(x-z_c)|)\; dx   \Big] 
\\
&\le  C \int_{\{|x-z_c|\le \delta\}} (1+|\log (x-z_c)|)\; dx 
\\
&\le  C  \delta (1+|\log \delta|).
\end{aligned}$$
Finally, the boundary terms can be treated, following the previous treatment as done in the case $k=0$.   
This completes the proof of the lemma.\end{proof}

\begin{proof}[Proof of Lemma \ref{lem-mErrAiry} ] The proof follows similarly, but more straightforwardly, from the above proof for the localized part of the Green function, upon recalling that 
$$
Err_A(x,z)=  \epsilon  \D_z^2 \dot z^{1/2} \dot z^{-1/2} \dz^2G(x,z) - 2 \alpha^2 \epsilon \dz^2 G(x,z). 
$$
We omit further details. \end{proof}

\newpage  
\section{Construction of the slow Orr modes}\label{sec-construction-phi1}


In this section, we iteratively construct an exact Orr-Sommerfeld solutions. Our construction 
starts with the Rayleigh solutions $\phi_{j,\alpha}$, constructed in Section \ref{sec-Rayleigh}.
Precisely, we obtain the following Proposition whose proof will be given at the end of the section, 
yielding an exact solution to the Orr-Sommerfeld equations.

\begin{prop}\label{prop-construction-exactphi1} For sufficiently small $\alpha, z_c, \epsilon$
 satisfying 
$z_c^3\ll \delta$, there exist exact solutions $\phi_1(z)$
and $\phi_2(z)$  in $X_2$ which solve the Orr-Sommerfeld equations
$$Orr(\phi_{j}) = 0,\qquad j = 1,2.$$
In addition, we can construct $\phi_1(z)$ and $\phi_2(z)$
 so that $\phi_{j}$ is close to $\phi_{j,\alpha}$ in $X_2$, 
 that is, $$ \| \phi_{j} - \phi_{j,\alpha}\|_{X_2} \le C\delta (1+|\log \delta|)(1+|z_c/\delta|)^{3/2}$$
for some positive constant $C$ independent of $\alpha,z_c, \epsilon$.   
\end{prop}


Next, we obtain the following lemma. 
\begin{lemma}\label{lem-analytic01} 
The slow Orr modes $\phi_{1,2}$ constructed in Proposition \ref{prop-construction-exactphi1} depends analytically in $c$, for $\I c>0$. 
\end{lemma}
\begin{proof} The proof is straightforward since the only ``singularities'' are of the forms: $\log(U-c)$, $1/(U-c)$, $1/(U-c)^2$, and $1/(U-c)^3$, which are of course analytic in $c$ when $\I c>0$.  
\end{proof}

\begin{remark}  \textup{In fact, it can be shown that the solutions $\phi_{1,2}$ can be extended $C^\gamma$-H\"older continuously on the axis $\{\I c =0\}$, for $0\le \gamma<1$. 
}\end{remark}

\subsection{Principle of the construction}


We now present the idea of the iterative construction. 
We start from the Rayleigh solutions $\phi_{j,\alpha}(z)$ constructed in Section \ref{sec-Rayleigh}.
Since they solve the Rayleigh equation exactly, we have 
\begin{equation}\label{Orr-1stapp} 
Orr(\phi_{j,\alpha}) =- \epsilon (\dz^2 - \alpha^2)^2 \phi_{j,\alpha}  = Diff(\phi_{j,\alpha}).
\end{equation}
Here we observe that the right hand side, denoted $O_1(z)$, is of order $\cO(\epsilon)$. 
However, it contains a singularity near the critical layer $z = z_c$ since $\phi_{j,\alpha}(z)$ 
has a singularity of order $(z-z_c)\log(z-z_c)$. We then apply the Airy solver 
to smooth out the singularity. Precisely, we denote 
$$ 
A_{j,0}: =  AirySolver(Diff(\phi_{j,\alpha})).
$$
By Proposition \ref{prop-boundAiryS}, 
$$
\|A_{j,0}\|_{X_2} \le C \|\phi_{j,\alpha}\|_{Y_4}\delta (1+|\log \delta|)(1+|z_c/\delta|)^{3/2},
$$
 which is small if $z_c^3\ll \delta$. 
 We then modify $\phi_{j,\alpha}$ by adding this corrector $A_{j,0}$. We introduce 
$$ 
\phi_{j,1}: = \phi_{j,\alpha} + A_{j,0}.
$$
We then have 
\begin{equation}\label{def-error1}
Orr(\phi_{j,1}) = O_{j,1}: =-  AiryErr(Diff(\phi_{j,\alpha})) + Reg (AirySolver(Diff(\phi_{j,\alpha}))), 
\end{equation}
which has no singularity other than $(z-z_c)\log(z-z_c)$ 
and of smaller order as compared to the right hand side of \eqref{Orr-1stapp}.  
By Proposition \ref{prop-boundAiryS}, and as $Reg$ is a simple multiplication by a bounded function,
$$
\| Orr(\phi_{j,1}) \|_{X_2} \le  C \|\phi_{j,\alpha}\|_{Y_4}\delta (1+|\log \delta|)(1+|z_c/\delta|)^{3/2}.
$$
We then approximately solve (\ref{def-error1}) by
$$
\psi_{j,2} =  - RaySolver_{\alpha} \Bigl( O_{j,1} \Bigr) .
$$
We thereby gain the control of two derivatives and an extra $z - z_c$ term thanks to Proposition
\ref{prop-exactRayS} and get 
$$
\| \psi_{j,2} \|_{Y_4} \le C \| O_{j,1} \|_{X_2} .
$$
We inductively proceed the above construction. 
Let us assume that we have constructed approximate solutions $\phi_{j,N}(z) \in X_2$, 
$j = 1,2$ and $N\ge 1$, so that $$
 Orr(\phi_{1,N})  = O_{j,N},
$$ 
with an error $O_{j,N}$ which is sufficiently small. 
We then introduce
$$
\psi_{j,N} := -  RaySolver_{\alpha} \Bigl( O_{j,N}  \Bigr) 
$$ 
in order to solve approximately the equation "in the interior of the domain".
Observe that by a view of \eqref{key-ids} and \eqref{eqs-RaySolver}
\begin{equation}\label{Orr-N}
Orr (\phi_{j,N} + \psi_{j,N})  = S_{j,N} = - Diff ( RaySolver_{\alpha}( O_{j,N} )).  
\end{equation}
We expect  $Diff ( RaySolver_{\alpha}( O_{j,N} ))$ to have a better error estimate, 
 precisely due to the extra $\eps$ present in the $Diff$ operator. However, 
 the Rayleigh equation contains a singular solution of the form $(z-z_c)\log (z-z_c)$, 
 and therefore $\psi_{1,N}$ admits the same singularity at $z = z_c$. 
 As a consequence, $Diff(\psi_{1,N})$ consists of singularities of orders $\log(z-z_c)$ 
 and $(z-z_c)^{-k}$, for $k=1,2,3$, and is large in the critical layer. To remove these singularities,
we then use the $Airy$ operator. More precisely, 
the $Airy$ operator  smoothes out the singularity inside the critical layer.  To do so, we introduce
$$
A_{j,N} = AirySolver(S_{j,N}) 
$$
which take care of the singularities in the critical layer.
We then define
$$
\phi_{1,N+1} := \phi_{1,N} + \psi_{j,N}+ A_{j,N}
$$
which  solves
 $$
Orr( \phi_{1,N+1} ) = AiryErr(S_{s,N}) - Reg (AirySolver(S_{s,N}))  .
$$ 
The point here is that although $S_{s,N}$ contains the mentioned singularity, $AirySolver(S_{s,N})$ and so $Orr( \phi_{1,N+1} )$ consist of no singularity, 
and furthermore the right hand side term $O_{j,N+1}$  has a better error as compared to $O_{j,N}$. To ensure the convergence, let us introduce the iterating operator $Iter(f)$ defined by 
\begin{equation}\label{def-Iter}
\begin{aligned}
Iter(f): =& AiryErr(Diff (RaySolver_\alpha(f)))   
\\& - Reg (AirySolver(Diff (RaySolver_\alpha(f)))).
\end{aligned}
\end{equation}
Then 
$$
O_{j,N+1} = Iter(O_{j,N}) .
$$
We shall prove the following key lemma which gives sufficient estimates on the $Iter(\cdot)$ operator. 

\begin{lemma}\label{lem-keyIter}
 For $f \in X_2$, the $Iter(\cdot)$ operator defined as in \eqref{def-Iter} 
 is a well-defined map from $X_2$ to $X_2$. Furthermore, there holds 
\begin{equation}\label{est-keyIter}
 \| Iter(f)\|_{X_2} \le C\delta (1+|\log \delta|)
  (1+|z_c/\delta|)^{3/2}  \|f\|_{X_2},
  \end{equation}
for some universal constant $C$. 
\end{lemma}

\begin{proof} Take $f \in X_2$. By Proposition \ref{prop-exactRayS}, 
$F(z): = RaySolver_\alpha(f)(z)$ is well-defined for all $z\in [0,1]$, and satisfies 
$$
\| F\|_{Y_4} \le C 
\| f\|_{X_2}.
$$
Furthermore, $\partial_z F (1) = 0$. Next, 
Proposition \ref{prop-boundAiryS} can be applied to get
\begin{equation} \label{keyest-ADF} 
\| AirySolver(Diff(F))\|_{X_2} \le C \|F\|_{Y_4}\delta (1+|\log \delta|)(1+|z_c/\delta|)^{3/2},
\end{equation}
and 
\begin{equation}\label{keyest-EDF} 
\| AiryErr(Diff(F))\|_{X_2} \le C \|F\|_{Y_4} \delta^2 (1+|\log \delta|).
\end{equation}
Combining these estimates together and recalling that 
$Reg(\phi): =- (\eps \alpha^4 + U'' + \alpha^2 (U-c) )\phi$ is simply a multiplication by a bounded function,
 the Lemma follows at once. 
\end{proof}

\begin{proof}[Proof of Proposition \ref{prop-construction-exactphi1}]


Using the previous Lemma we construct by iteration functions $\phi_{j,N}$ such that 
\begin{equation}\label{def-phiNa} 
Orr(\phi_{j,N})(z) = O_{j,N}(z),
\end{equation}
where the error $O_{j,N}(z)$ satisfies
$$ 
\|O_{j,N}\|_{X_2} \le C \Big[C\delta (1+|\log \delta|) 
(1+|z_c/\delta|)^{3/2}\Big]^N,
$$
and where $\phi_{j,N}$ satisfy the same bound in $Y_4$.
We then define $\phi_j$, which satisfy $Orr(\phi_j) = 0$, by the following convergent serie
\beq \label{def-phi1}
 \phi_j(z) =
  \phi_{j,\alpha}(z) +  AirySolver(Diff(\phi_{j,\alpha})) 
  + \sum^{+\infty}_{n\ge 1} \Big[ \psi_{j,n}   + AirySolver \Bigl( Diff (\phi_{j,n}) \Bigr)\Big] 
 \eeq
with $\psi_{j,n}= -RaySolver_{\alpha}( O_{j,n})$ and $ O_{j,n+1}(z): = Iter (O_{j,n})(z),$
 in which $O_{j,1} $ is defined as in \eqref{def-error1} 
 and the iterated operator $Iter(\cdot)$ is defined as in \eqref{def-Iter}, 
  the exact Rayleigh solver $RaySolver_{\alpha}(\cdot) $ is constructed as in Proposition \ref{prop-exactRayS}.

\end{proof}


\subsection{First order expansion of $\phi_{1,2}$ at $z=0$} 


In this paragraph we explicitly compute the boundary contribution of the first terms 
in the expansion of $\phi_{j}(0)$, $j = 1,2$. 
  We shall use the estimates obtained in Lemma \ref{lem-phi1alpha}. 
   In study of the dispersion relation, we are interested in various ratios between these solutions.
    For convenience, let us define  
\begin{equation}\label{def-K1}
K_1: =  \frac{\phi_1(0)}{\partial_z\phi_1(0)}.
\end{equation} 
In this section, we will prove the following Lemma.

\begin{lemma}\label{lem-ratioK1} Let $\phi_1$ be defined as in \eqref{def-phi1}, and let $K_1$ be defined as in \eqref{def-K1}. For small $z_c, \alpha, \delta$ with $z_c\approx \alpha^2$ and $\delta \lesssim z_c$, there hold
\begin{equation}\label{keyratio-K1}
\begin{aligned}
K_1 &=  \frac{U_0 - c}{U'_0} + \frac{\alpha^2 }{|U'_0|^2}\int_0^1 (U-U_0)^2 \; dx +\cO( \alpha^2 \delta (1+|\log \delta|)(1+|z_c/\delta|)^{3/2}),\quad
\\
\I  K_1 &= - \frac{ \I c}{U'_0} \Big ( 1+  \cO(\alpha^2\log \alpha) \Big ) +\cO( \alpha^2 \delta (1+|\log \delta|)(1+|z_c/\delta|)^{3/2}). 
  \end{aligned}
\end{equation}
In particular, $K_1 = \cO(\alpha^2)$. 
\end{lemma}

In the above lemma, the assumption $z_c\approx \alpha^2$ is only used to simplify the claimed estimates. 
Such an assumption will be verified in Section \ref{sec-expdisp}. 
The proof of Lemma \ref{lem-ratioK1} follows directly from several lemmas, obtained below in this section, 
together with Lemma \ref{lem-phi1alpha}. 
 We first give the boundary estimates on  $A_{j,0}(z)$

\begin{lemma} Let $A_{1,0}(z)  = AirySolver(Diff(\phi_{1,\alpha})) $. There hold
\begin{equation}\label{A10-bound}\begin{aligned}
|A_{1,0}(z;\epsilon,c)| &\le  C \delta^3 + C\alpha^2 \delta (1+|\log \delta|)(1+|z_c/\delta|)^{3/2}
\\
|\partial_zA_{1,0}(0; \epsilon,c)| &\le C \delta^2 + C z_c^{-1} \alpha^2 \delta (1+|\log \delta|)(1+|z_c/\delta|)^{3/2}
\end{aligned}\end{equation}
\end{lemma}
\begin{proof}
 We recall that $\phi_{1,\alpha} = U-c + \cO(\alpha^2)$, hence its leading order is smooth.
 Indeed, we have  
$$
Diff(U-c) = \cO(\delta^3).
$$
Proposition \ref{prop-exactAiry} then yields 
$$
\partial_z^k AirySolver(Diff(U-c)) (z) = \cO(\delta^{3-k}),
$$ 
for $k = 0,1$.
Next, the $\cO(\alpha^2)$ term is of order $\alpha^2$ in $Y_4$, hence 
$$
\| AirySolver( \cO(\alpha^2)) \|_{X_2} \le  C z_c^{-1} \alpha^2 \delta (1+|\log \delta|)(1+|z_c/\delta|)^{3/2}.
$$
The definition of the $X_2$ norm ends the proof of the Lemma.
\end{proof}

\begin{proof}[Proof of  Lemma \ref{lem-ratioK1}]

Let us recall that $\psi_{j,1}= -RaySolver_{\alpha}( O_{j,1})$, with 
$$
O_{j,1} (z)=-  AiryErr(Diff(\phi_{j,\alpha})) + Reg (AirySolver(Diff(\phi_{j,\alpha}))).
$$
Again, by a view of \eqref{keyest-ADF} and \eqref{keyest-EDF}, the error term $O_{j,1}$ is of the same order as that of $A_{j,0}$, and so is $\psi_{j,1}$.

Combining the above estimates, we have obtained 
\begin{equation}\label{expansion-phi1}\begin{aligned}
\phi_1(0) &= U_0 - c + \frac{\alpha^2 }{U'_0}\int_0^1 (U-c)^2 \; dx \\&\quad + \cO(\alpha^2 z_c \log z_c) +\cO( \delta^3 + \alpha^2 \delta (1+|\log \delta|)(1+|z_c/\delta|)^{3/2})
\end{aligned}
\end{equation}
for small $z_c, \alpha,\epsilon$. As for derivative, we also have
\begin{equation}\label{expansion-dphi1}
\partial_z\phi_1(0) =  U'_0 + \cO(\alpha^2\log z_c) + \cO( \delta^2 +  z_c^{-1} \alpha^2 \delta (1+|\log \delta|)(1+|z_c/\delta|)^{3/2})
\end{equation}
Now, under the assumption of the lemma that $z_c\approx \alpha^2$ and $\delta\lesssim z_c$, the above expansions are simplified to
\begin{equation}
\label{expansion-simphi1}\begin{aligned}
 \phi_1(0) &= U_0 - c + \frac{\alpha^2 }{U'_0}\int_0^1 (U-c)^2 \; dx +\cO( \alpha^2 \delta (1+|\log \delta|)(1+|z_c/\delta|)^{3/2})
 \\
 \partial_z\phi_1(0) &=  U'_0 +\cO( \delta (1+|\log \delta|)(1+|z_c/\delta|)^{3/2})
 \end{aligned}\end{equation}
 The claimed estimate of $\phi_1/\partial_z\phi_1$ now follows easily, upon noting that $U_0-c = \cO(z_c)$. 

Finally, let us study the imaginary part of $\phi_1/\partial_z\phi_1$. It is clear from the above expansions that 
$$\I \Big(\frac {\phi_1}{\partial_z\phi_1} \Big) = - \frac{ \I c}{U'_0} \Big ( 1+  \cO(\alpha^2\log \alpha) \Big ) +\cO( \alpha^2 \delta (1+|\log \delta|)(1+|z_c/\delta|)^{3/2}).$$

This proves the lemma.
 \end{proof}


\subsection{First order expansion of $\phi_{1,2}$ at $z=1$} 

Similarly to the previous section, we are interested in the ratio: \begin{equation}\label{def-K2}K_2: =  \frac{\partial_z\phi_1(1)}{\partial_z\phi_2(1)}.\end{equation} 
We shall prove the following lemma. 

\begin{lemma}\label{lem-ratioK2} Let $\phi_j$ be defined as in \eqref{def-phi1}, and let $K_2$ be defined as in \eqref{def-K2}. For small $z_c, \alpha, \delta$, there hold
\begin{equation}\label{keyratio-K2}
\begin{aligned}
K_2 &= \alpha^2 (1+ \cO(\alpha^2))  \int_0^1 (U-U_0)^2\; dx .
\end{aligned}
\end{equation}
\end{lemma}
\begin{proof} We recall that 
$$ \phi_{j}(z) = \phi_{j,\alpha} + AirySolver(Diff(\phi_{j,\alpha})) + \sum_{n\ge 1} \Big[ \psi_{j,n}   + AirySolver \Bigl( Diff (\psi_{j,n}\Bigr)\Big] 
 $$
in which $\psi_{j,n}= -RaySolver_{\alpha}( O_{j,n})$. By definition of the $RaySolver_\alpha(\cdot)$ operator, together with the assumption that $\partial_z \phi_{1,0}(1) = U'(1) = 0$, it follows directly that 
$$\begin{aligned}
\partial_z \phi_{j,\alpha}(1) &= 0,\qquad \partial_z \psi_{j,n} (1) = 0,
\end{aligned}$$
for $j = 1,2$ and $n\ge 1$. In addition, the term $AirySolver ( Diff (\psi_{j,n} )$ is of a higher order. It thus suffices to give estimates on the derivative of 
$$A_{j,0}(z) = AirySolver(Diff(\phi_{j,\alpha})) (z).$$
We recall that 
$$\begin{aligned}
\phi_{1,\alpha} &= (U-c)  \Big[ 1 +\alpha^2 \int_0^z \phi_{1,0}\phi_{2,0}\; dx - \alpha^2 \frac{\phi_{2,0}}{U-c}\int_0^z (U-c)^2 \; dx\Big] \\&\quad + \alpha^2 \phi_{2,0}\int_0^1 (U-c)^2\; dx + \cO(\alpha^4)\phi_{2,0}
\\
\phi_{2,\alpha}& = \phi_{2,0} + \cO(\alpha^2),
\end{aligned}$$
in which $\phi_{2,0}(z)$ contains a $(z-z_c)\log (z-z_c)$ singularity near the critical layer $z = z_c$. We note that the first bracket term in $\phi_{1,\alpha}$ is regular near $z = z_c$, and thus can be neglected when convoluted with $AirySolver(Diff(\cdot))$ due to the extra factor of $\epsilon$ in the $Diff(\cdot)$ operator. We are only concerned with the singular terms, which occur at order $\cO(1)$ in $\phi_{2,\alpha}$, whereas at order $\cO(\alpha^2)$ in $\phi_{1,\alpha}$. Precisely, we have the following expansions: 
$$
\begin{aligned}
|\partial_zA_{1,0}(1)| & = \alpha^2(1+ \cO(\alpha^2))  K_0 \int_0^1 (U-c)^2\; dx 
\\|\partial_zA_{2,0}(1)| & = K_0 (1+ \cO(\alpha^2)),
\end{aligned}$$
in which $K_0: = \partial_z AirySolver(Diff(\phi_{2,0}))_{\vert_{z=1}}.$ Putting these together proves the lemma. 
\end{proof}

\newpage

\section{Construction of the fast Orr modes}\label{sec-construction-phi3}


In this section we provide a similar construction
 to that obtained in Proposition \ref{prop-construction-exactphi1}. The construction will begin with the Airy solutions. Precisely, let us introduce
\begin{equation}\label{def-phi30}
\phi_{3,0}(z) : =  Ai(2,\delta^{-1}\eta(z)) ,\qquad \phi_{4,0}(z) : =  Ci(2,\delta^{-1}\eta(z)) .
\end{equation} 
 Here $Ai(2,\cdot)$ and $Ci(2,\cdot)$ are the second primitive of the Airy solutions $Ai(\cdot)$ and $Ci(\cdot)$, respectively, and $\eta(z)$ denotes the Langer's variable
\begin{equation}\label{Lg-transform}
\delta = \Bigl( { \eps \over U'_c} \Bigr)^{1/3} , \qquad\quad \eta(z)
= \Big[ \frac 32 \int_{z_c}^z \Big( \frac{U-c}{U'_c}\Big)^{1/2} \; dz \Big]^{2/3}.
\end{equation}
We recall that as $Z\approx (z-z_c)/\delta \to \pm\infty$, the Airy solution $Ai(2, e^{i \pi/6}Z)$ behaves as $e^{\mp \frac {\sqrt{2}}{3} |Z|^{3/2}}$, whereas $Ci(2, e^{i \pi/6}Z)$ is of order $e^{\pm \frac {\sqrt{2}}{3} |Z|^{3/2}}$. Let us also recall that the critical layer is centered at $z = z_c$ and has a typical
size of $\delta$. Inside the critical layer, the Airy functions play a crucial role.

\begin{prop}\label{prop-construction-phi3} For $\alpha, \delta$ sufficiently small,  there are exact solutions $\phi_{j}(z)$, $j = 3,4$, solving the Orr-Sommerfeld equation
$$ 
Orr(\phi_{j}) = 0, \qquad j = 3,4.
$$
In addition, we can construct $\phi_{j}(z)$ so that $\phi_{j}(z)$ is approximately close to $\phi_{j,0}(z)$ in the sense that 
\begin{equation}\label{est-phi3-Orr} 
 |\phi_{j} (z)  - \phi_{j,0} (z) | \le C  \delta , \qquad j = 3,4,
 \end{equation}
 for some fixed constants $\eta,C$. 
\end{prop}

From the construction, we also obtain the following lemma. 

\begin{lemma}\label{lem-analytic03} 
The fast Orr mode $\phi_{3,4}(z)$ constructed in Proposition \ref{prop-construction-phi3} 
depends analytically in $c$ with $\I c \not = 0$. 
\end{lemma}

\begin{proof} This is simply due to the fact that both Airy function 
and the Langer transformation \eqref{Lg-transform} are analytic in their arguments. 
\end{proof}


\subsection{Iterative construction}

Let us prove Proposition \ref{prop-construction-phi3} in this section.
\begin{proof}[Proof of Proposition \ref{prop-construction-phi3}]
We start with $\phi_{3,0}(z)  =  Ai(2,\delta^{-1}\eta(z))$. We recall that 
$\phi_{3,0}(z)$ satisfies
$$ |\phi_{3,0}(z)| \le C_0 (1+|Z|)^{-5/4}   e^{- \sqrt{2|Z|} Z/3} ,$$
uniformly for all $z\in [0,1]$. Direct calculations yield 
$$\begin{aligned}
 Airy(\phi_{3,0}):= &~\eps \delta^{-1}\eta^{(4)} Ai(1,Z) + 4 \eps \delta^{-2} \eta' \eta^{(3)} Ai(Z) + 3\eps \delta^{-2} (\eta'')^2 Ai(Z)+ \eps \delta^{-4} (\eta')^4 Ai''(Z) \\&+ 6 \eps \delta^{-3} \eta '' (\eta')^2 Ai'(Z)  - (U-c) \Big[\eta'' \delta^{-1}Ai(1,Z) + \delta^{-2}(\eta')^2 Ai(Z)\Big],
\end{aligned}
$$ with $Z = \delta^{-1}\eta(z)$. Let us first look at the leading terms with a factor of $\eps \delta^{-4}$ and of $(U-c)\delta^{-2}$. Using the facts that $\eta' = 1/\dot z$, $\delta ^3= \eps/U_c'$, and $(U-c)\dot z^2 = U_c' \eta(z)$, we have 
$$\begin{aligned}  \eps \delta^{-4} (\eta')^4 Ai''(Z) &- \delta^{-2}(\eta')^2  (U-c)Ai(Z) \\
&=  \eps \delta^{-4} (\eta')^4 \Big[Ai''(Z) - \delta^{2}\eps^{-2}  (U-c)\dot z^2 Ai(Z)\Big]  \\&=  \eps \delta^{-4} (\eta')^4 \Big[Ai''(Z) - Z Ai(Z)\Big] = 0.
\end{aligned}$$
The next terms in $Airy(\phi_{3,0})$ are
$$\begin{aligned}
6 \eps \delta^{-3} \eta '' (\eta')^2 Ai'(Z)  &- (U-c)\eta'' \delta^{-1}Ai(1,Z) \\&= \Big[6 \eta '' (\eta')^2 U_c' Ai'(Z)  - Z U_c' \eta'' (\eta'^2)Ai(1,Z)\Big]
\\&= \eta '' (\eta')^2 U_c'  \Big[ 6Ai'(Z)  - Z Ai(1,Z) \Big] ,\end{aligned}$$ 
which is of order $\cO(1)$. The rest is of order $\cO(\eps^{1/3})$ or smaller. That is, we obtain
 $$ Airy(\phi_{3,0}) =  \eta '' (\eta')^2 U_c'  \Big[ 6Ai'(Z)  - Z Ai(1,Z) \Big] + \cO(\eps^{1/3}).$$
This shows that $Airy(\phi_{3,0})(z)$ is very localized and depends primarily on the fast variable $Z$ as $Ai(\cdot)$ does. Furthermore, we have 
$$
|Airy(\phi_{3,0})(z)|\le C (1+|Z|)^{1/4}   e^{- \sqrt{2|Z|} Z/3} $$
for some constant $C$. By the identity \eqref{key-ids}, it follows that
$$Orr(\phi_{3,0}) = I_0(z): = Reg(\phi_{3,0}) + \eta '' (\eta')^2 U_c'  \Big[ 6Ai'(Z)  - Z Ai(1,Z) \Big] + \cO(\eps^{1/3}), 
$$in which $Reg(\phi): = -  (\eps \alpha^4 + U'' + \alpha^2 (U-c) ) \phi.$ In addition, the bound on $\phi_{3,0}$ and on $Airy(\phi_{3,0})$ yields 
\begin{equation}\label{I0-bound} 
| I_0(z)| \le C_0(1+|Z|)^{1/4}   e^{- \sqrt{2|Z|} Z/3} .
\end{equation}

To obtain a better error estimate, let us introduce $ \phi_{3,1}(z): = \phi_{3,0}(z)  - AirySolver_\infty (I_0) (z)$. We then get  
$$ Orr(\phi_{3,1}) = I_1(z): = - Reg (AirySolver_\infty (I_0)) (z),$$
in which by a view of Lemma \ref{lem-locConvAiry} and Proposition \ref{prop-exactAiry-loc}, $I_1(z)$ is of order $\cO(\delta)$ smaller than that of $I_0(z)$. Precisely, we have 
\begin{equation}\label{I1-bound}|I_1(z)| \le C\delta  (1+|Z|)^{-7/4}   e^{- \sqrt{2|Z|} Z/3}   + C\delta,\end{equation}
in which the terms on the right are due to the convolution with the localized and non-localized part of the Green function of the $Airy$ operator, respectively. In particular, $I_1 = \cO(\delta)$. We then inductively introduce 
$$I_{n+1} := Iter(I_n)$$
where $Iter(\cdot)$ is defined as in \eqref{def-Iter}. 
Proposition \ref{prop-construction-exactphi1} ensures the convergence of the series 
\begin{equation}
\label{series-phi3n}\phi_{3,N}(z) = \phi_{3,0}(z)  - AirySolver_\infty (I_0) (z) + \sum_{n=1}^{+\infty}
 \Big[ \psi_{n}   + AirySolver \Bigl( Diff (\psi_{n}\Bigr)\Big] \end{equation}
in which $\psi_{n}: = -RaySolver_\alpha(I_n)$. The limit of $\phi_{3,N}$ as $N\to \infty$ yields the third Orr modes as claimed.

A similar construction applies for $\phi_{4,0} = \gamma_4 Ci(2,\delta^{-1}\eta(z))$, since both $Ai(2,\cdot)$ and $Ci(2,\cdot)$ solve the same primitive Airy equation. 

 This proves the proposition. 
%
%
\end{proof}

\subsection{First order expansion of $\phi_3$ at $z=0$}
By the construction in Proposition \ref{prop-construction-phi3}, we obtain the following first order expansion of $\phi_3$ at the boundary $z=0$:
 $$ \phi_3(0) = 1 + \cO(\delta), \qquad \partial_z \phi_3(0) =  \delta^{-1} {Ai(1,\delta^{-1} \eta(0)) 
\over Ai(2,\delta^{-1} \eta(0)) }  (1+\cO(\delta)).$$
In the study of the linear dispersion relation, we are interested in the ratio $\partial_z \phi_3 / \phi_3$. Again, for convenience, let us introduce
\begin{equation}\label{def-K3}K_3: = {\phi_{3}(0) \over \partial_z \phi_{3}(0)} .\end{equation}
The above estimates yield 
\begin{equation}\label{ratio-phi3}
K_3
 = \delta C_{Ai}(\delta^{-1} \eta(0)) ( 1 + \cO(\delta)),
\qquad \mbox{with}\quad C_{Ai} (Y):= 
 {Ai(2,Y) \over Ai(1,Y) } .
\end{equation}
The following lemma is crucial later on to determine instability. 
\begin{lemma}\label{lem-ratioK3} Let $\phi_3$ be the Orr-Sommerfeld solution constructed in Proposition \ref{prop-construction-phi3}, and let $K_3$ be defined as in \eqref{def-K3}. There holds
\begin{equation}\label{keyratio-K3}
K_3 = - e^{\pi i/4} |\delta| |z_c/\delta|^{-1/2} (1+\cO(|z_c/\delta|^{-3/2}))
\end{equation}
as long as $z_c/\delta$ is sufficiently large. In particular, the imaginary part of $\phi_3 / \partial_z\phi_3$ becomes negative when $z_c/\delta$ is large. In addition, when $z_c/\delta = 0$, 
\begin{equation}\label{keyratio-phi3-stable} K_3  = 3^{1/3} \Gamma(4/3) |\delta| e^{5i\pi/6 },
\end{equation}
for $\Gamma(\cdot)$ the usual Gamma function. 
\end{lemma}

Here, we recall that $\delta =  e^{-i \pi / 6} (\alpha R U_c')^{-1/3}$, and from the estimate \eqref{est-eta}, $\eta(0) = - z_c + \cO(z_c^2)$. Therefore, we are interested in the ratio $C_{Ai}(Y)$ for complex $Y = - e^{i \pi /6}y$, for $y$ being in a small neighborhood of $ \RR^+$. Without loss of generality, in what follows, we consider $y \in \RR^+$. Lemma \ref{lem-ratioK3} follows directly from the following lemma. 

\begin{lemma}\label{lem-CAi} Let $C_{Ai}(\cdot)$ be defined as above. Then, $C_{Ai}(\cdot)$ is uniformly bounded on the ray $Y = e^{7i\pi/6} y$ for $y \in \RR^+$. In addition, there holds 
$$ C_{Ai}(- e^{i \pi /6} y)  =  -  e^{ 5i \pi / 12} y^{-1/2} (1+\cO(y^{-3/2}))  $$
for all large $y\in \RR^+$. At $y = 0$, we have 
$$ C_{Ai} (0) = - 3^{1/3} \Gamma(4/3).$$
\end{lemma}  
\begin{proof} 
Thus, using asymptotic behavior of $Ai$,
yields 
$$ C_{Ai}(Y)  =  - Y^{-1/2} (1 + \cO(|Y|^{-3/2}))$$
for large $Y$. This proves the estimate for large $y$. 
 The value at $y = 0$ is easily obtained from those of $Ai(k,0)$. 
This completes the proof of the lemma. \end{proof}

\subsection{First order expansion of $\phi_{3,4}$ at $z=1$}
As will be clear in the next section, we shall need to estimate the values of derivatives of $\phi_{3,4}$ at $z=1$ as well as the ratio 
\begin{equation}\label{def-K4} K_4 := \frac{\phi_4'(1)}{\phi_4'''(1)}.\end{equation}
\begin{lemma}\label{lem-bdryphi34} Let $\phi_{3,4}$ be the Orr-Sommerfeld solution constructed in Proposition \ref{prop-construction-phi3}. There hold
\begin{equation}\label{est-bdryphi34}
\begin{aligned}
\partial_z^k \phi_3(1) &\lesssim \delta^2 (1+|\log \delta|) (1+|z_c/\delta|)^{3/2} 
\\
\partial_z^k \phi_4(1) &\gtrsim \cO(e^{1/\sqrt {|\epsilon|}})
\end{aligned}\end{equation}
in $L^\infty$, for $k = 1,3$. 
\end{lemma}

\begin{proof} We recall that the construction in Proposition \ref{prop-construction-phi3} gives
\begin{equation}
\label{series-phi34}\phi_{j}(z) = \phi_{j,0}(z)  - AirySolver_\infty (I_0) (z) + \sum_{n\ge 1} \Big[ \psi_{n}   + AirySolver \Bigl( Diff (\psi_{n}\Bigr)\Big] \end{equation}
with  $\psi_{n}: = -RaySolver_\alpha(I_n)$. Let us give estimates for $\phi_3$. We recall that $\phi_{3,0}(z) = \gamma_3 Ai(2,\delta^{-1}\eta(z))$, and  $I_0(z)$ satisfies \eqref{I0-bound}. Thanks to \eqref{boundAik1}, the claimed estimate for $\phi_{3,0}(1)$ and its derivatives follows easily, upon noting that $\eta(1) \approx 1$ and $\delta \approx \epsilon^{1/3}$. Next, let $k\ge 1$. By definition (see \eqref{eqs-Airyp-S} and \eqref{def-GreenAiry2}), we have
$$ \partial_z^k AirySolver (I_0) (1) = \int_0^1 \partial_z^k\widetilde G(x,1) I_0(x) \; dx .$$
Now, since $I_0(x)$ is very localized, a very similar calculation as done in Lemma \ref{lem-locConvAiry} yields  
$$ |\partial_z^k AirySolver (I_0) (1)| \lesssim  \delta e^{-1/\sqrt \epsilon},\qquad \forall ~k\ge 1.$$
As for the next term $\psi_n$ in \eqref{series-phi3n}, we observe that $\partial_z^k\psi_{n} (1) = - \partial_z^kRaySolver_\alpha(I_n)(1) = 0$, by definition of the $RaySolver_\alpha$ operator. In addition, we recall that $I_1 = \cO(\delta)$ and so $\psi_n = \cO(\delta)$, for $n \ge 1$. Finally, as in the above estimate, we obtain 
$$ \partial_z^k AirySolver \Bigl( Diff (\psi_{n}\Bigr) (1) =  \int_0^1 \partial_z^k\widetilde G(x,1) Diff (\psi_{n})  \; dx , \qquad k\ge 1.$$
Lemma \ref{lem-mAiry} then yields 
$$ | \partial_z^k AirySolver \Bigl( Diff (\psi_{n}\Bigr) (1)  | \lesssim \delta^2 (1+|\log \delta|) (1+|z_c/\delta|)^{3/2}$$
for $k\ge 1$ and $n\ge 1$. This proves the claimed estimate for $\phi_3$ at the boundary $z=1$.

Similarly, as for $\partial_z^k\phi_4(1)$, we have started the expansion with $\phi_{4,0}(z) = \gamma_4 Ci(2,\delta^{-1}\eta(z))$, which is of order $e^{1/\sqrt\epsilon}$ at the boundary $z=1$. This yields the claimed lower bound for the derivatives of $\phi_4$ at the boundary. 
\end{proof}


\begin{lemma}\label{lem-ratioK4} Let $\phi_4$ be the Orr-Sommerfeld solution constructed in Proposition \ref{prop-construction-phi3}, and let $K_4$ be defined as in \eqref{def-K4}. There holds
\begin{equation}\label{keyratio-K4}
 K_4= \cO(\delta^2).
\end{equation}
\end{lemma}
\begin{proof} Indeed, up to an error of order $\cO(\delta^2)$, the fast mode $\phi_{4}(z)$ primarily depends on the fast variable $Z = \eta(z) /\delta$. This shows $\partial_z \approx 1/\delta$, and the estimate for the ratio thus follows. 
\end{proof}

\newpage


\section{Study of the dispersion relation}\label{sec-disp-relation}



\subsection{Linear dispersion relation}


A solution of \eqref{OS1}--\eqref{OS3} is a linear combination of the slow solutions $\phi_{1,2}$ that link with the Rayleigh solutions and the localized solutions $\phi_{3,4}$ that link with the Airy functions. Let us then introduce an exact Orr-Sommerfeld solution of the form
\begin{equation}\label{bl-phiN} \phi: = A_1 \phi_{1} + A_2 \phi_2 + A_3 \phi_3 + A_4\phi_4,\end{equation}
for some parameters $A_j = A_j(\epsilon,c)$,
 where $\phi_{1,2} = \phi_{1,2}(z;\eps,c)$ and $\phi_{3,4}=\phi_{3,4}(z;\eps,c)$ 
 are constructed in Propositions \ref{prop-construction-exactphi1} and \ref{prop-construction-phi3}, respectively.
  It is clear that $\phi(z)$ is an exact solution to the Orr-Sommerfeld equation. 
  The boundary conditions \eqref{OS2}-\eqref{OS3} at $z=0,1$ yield that the determinant 
\begin{equation}\label{disp}
W_0 (\epsilon, c): = \det \begin{pmatrix} \phi_1(0) & \phi_2(0) & \phi_3(0) & \phi_4(0) 
\\ \phi'_1(0) & \phi'_2(0) & \phi'_3(0) & \phi'_4(0) 
\\
\phi'_1(1) & \phi'_2(1) & \phi'_3(1) & \phi'_4(1) 
\\
\phi'''_1(1) & \phi'''_2(1) & \phi'''_3(1) & \phi'''_4(1) 
\end{pmatrix}  = 0.\end{equation}
This identity represents an eigenvalue dispersion relation, from which we shall obtain the existence of unstable eigenvalue $c$ with $\I c > 0$ for a certain range of parameter $\alpha = \alpha(\epsilon)$. 

We first relate this dispersion relation to those ratios $K_j$, $j=1,...,4, $ defined previously in \eqref{def-K1}, \eqref{def-K2}, \eqref{def-K3}, and \eqref{def-K4}, respectively.  Indeed, by dividing the last column in the above matrix by $\phi'''_4(1)$ and recalling from Lemma \ref{lem-bdryphi34} that at $z=1$,  the derivatives of $\phi_4(z)$ are of order $e^{1/\sqrt{\epsilon}}$, the last column in the determinant can be replaced by 
$$ \begin{pmatrix} 0\\ 0 \\ K_4 \\1 \end{pmatrix} + \cO(e^{-1/\sqrt \epsilon}), \qquad \mbox{with} \quad K_4 = \frac{\phi_4'(1)}{\phi_4'''(1)}.$$ 
Similarly, derivatives of $\phi_3(z)$ are of order $e^{-1/\sqrt\epsilon}$ at $z=1$. This shows that the third column in the above determinant can be replaced by 
$$ \begin{pmatrix} K_3\\ 1 \\ 0\\0 \end{pmatrix} + \cO(\delta^2 (1+|\log \delta|) (1+|z_c/\delta|)^{3/2}), \qquad \mbox{with} \quad K_3 = \frac{\phi_3(0)}{\phi_3'(0)}.$$ 

In addition, by a view of Lemma \ref{lem-ratioK4}, we have $K_4 = \cO(\delta^2)$. This proves that the relation $W_0(\epsilon, c) = 0$ reduces to 
$$W_1(\epsilon, c): = \det \begin{pmatrix} \phi_1(0) & \phi_2(0) & K_3 
\\ \phi'_1(0) & \phi'_2(0) & 1 
\\
\phi'_1(1) & \phi'_2(1) & 0 
\end{pmatrix}  = \cO(\delta^2 (1+|\log \delta|) (1+|z_c/\delta|)^{3/2}) .$$
This leads to a new dispersion relation:
\begin{equation*} 
K_3 = \frac{\phi_1(0) - K_2\phi_2(0)}{\phi_1'(0) - K_2 \phi_2'(0)} + \cO(\delta^2 (1+|\log \delta|) (1+|z_c/\delta|)^{3/2}),
\end{equation*} 
in which $K_2$ is defined as in \eqref{def-K2}. Notice that $K_2 = \cO(\alpha^2), \phi_1'(0) \approx U'_0\not =0,$ $\phi_{2,0}(0) \approx 1/U'_0$, and $\phi_2'(0) \approx \log z_c$. Upon recalling that $K_1 = \phi_1(0) / \phi_1'(0) = \cO(\alpha^2)$, the above dispersion relation is further reduced to 
\begin{equation} \label{new-disp}
K_3 = K_1 - \frac{K_2}{|U'_0|^2} + \cO(\alpha^4 \log \alpha) +  \cO(\delta^2 (1+|\log \delta|) (1+|z_c/\delta|)^{3/2}).
\end{equation} 


\subsection{Ranges of $\alpha$}


When $\eps=0$ 
our Orr-Sommerfeld equation simply becomes the Rayleigh equation, whose solutions indicate that $c(\alpha,0) = U(0) + \cO(\alpha^2)$ 
and the critical layer $z_c(\alpha,0) \approx \alpha^2$. Thus, when $\eps>0$, 
we expect that 
$(c(\alpha,\eps), z_c(\alpha,\eps)) \to (U(0),0)$ as $(\alpha,\eps) \to 0$ (which will be proved shortly). 
 In addition, as suggested by physical results (see, e.g., \cite{Reid,Schlichting}),
 and as will be proved below, 
  for instability, we would search for $\alpha$ between $(\alpha_\mathrm{low}(R), \alpha_\mathrm{up}(R))$ with 
$$ 
\alpha_\mathrm{low}(R) \approx  R^{-1/7}, \qquad\qquad \alpha_\mathrm{up}(R) \approx R^{-1/11},
$$ 
for sufficiently large $R$. These values of $\alpha_j(R)$ form lower and upper branches of the marginal (in)stability curve for the shear profile $U$. More precisely, we will show that there is a critical constant $A_{1c}$ so that with $\alpha_\mathrm{low}(R) = A_1 R^{-1/7}$, the imaginary part of $c$ turns from negative (stability) to positive (instability) when the parameter $A_1$ increases across $A_1=A_{1c}$. Similarly, there exists an $A_{2c}$ so that with $\alpha = A_2 R^{-1/11}$, $\I c$ turns from positive to negative as $A_2$ increases across $A_2 = A_{2c}$. In particular, we obtain instability in the intermediate zone: $\alpha \approx R^{-\beta}$ for $1/11<\beta<1/7$. 

We note that the ranges of $\alpha$ restrict the absolute value of $\delta = (\eps/U_c')^{1/3}$ to lie between $\delta_2$ and $\delta_1$, with $\delta_1 \approx \alpha^2$ and $\delta_2 \approx \alpha^{10/3}$, respectively. In particular, $\delta \lesssim \alpha^2$. Therefore, in the case $\alpha \approx \alpha_\mathrm{low}(R)$, the critical layer is accumulated on the boundary, and thus the fast-decaying mode in the critical layer plays a role of a boundary sublayer; 
in this case, the mentioned Langer transformation plays a crucial role.
 In the latter case when $\alpha \approx \alpha_\mathrm{up}(R)$,
 the critical layer is well-separated from the boundary at $z=0$. 

In the next subsections, we shall prove the following proposition, partially confirming the physical results. 

\begin{proposition}\label{prop-Imc} For $R$ sufficiently large, we show that $\alpha_\mathrm{low}(R) = A_{1c} R^{-1/7}$ and $\alpha_\mathrm{up}(R) = A_{2c} R^{-1/11}$, for some critical constants $A_{1c}, A_{2c}$, are indeed the lower and upper marginal branch for stability, respectively. In all cases of instability: $\alpha = A R^{-\beta}\in (\alpha_\mathrm{low}(R), \alpha_\mathrm{up}(R))$, there holds \begin{equation}\label{bound-Imc}\begin{aligned}
\I c \quad \approx \quad A^{-3/2}R^{(3\beta -1)/2},
\end{aligned}\end{equation}
and in particular, we obtain the growth rate
\begin{equation}\label{growth} \alpha \I c \quad \approx \quad R^{-(1-\beta)/2},\end{equation} 
with $\beta \in [\frac {1}{11}, \frac{1}{7}]$. 
\end{proposition}


\subsection{Expansion of the dispersion relation}\label{sec-expdisp}

In this section, we shall study in detail the expansion of the new dispersion relation derived in \eqref{new-disp}. First, we observe that when $\delta \gtrsim \alpha^2$, the last term in \eqref{new-disp} can be absorbed into $\alpha^4 \log \alpha$. This reduces the dispersion relation to 
\begin{equation} \label{new-disp1}
K_3 = K_1 - \frac{K_2}{|U'_0|^2} + \cO(\alpha^4 \log \alpha).
\end{equation} 
with $K_1 = U_0-c + \cO(\alpha^2)$, $K_2 = \cO(\alpha^2)$, and $K_3\approx \delta (1+|z_c/\delta|)^{-1/2} $. Hence as $\alpha, \eps, \delta \to 0$, the eigenvalue $c$ converges to $U_0$ with
\begin{equation}\label{bound-zc}
|U_0-c| = \cO(\alpha^2) + \cO(\delta (1+ |z_c/\delta|)^{-1/2} ).\end{equation} It then follows from the Taylor's expansion: $c = U(z_c) = U_0 + U'_0 z_c + \cO(z_c^2)$ that $z_c = \cO(\alpha^2)$.

Next, we give the existence of $c$ for small $\alpha,\epsilon$. 

\begin{lemma}\label{lem-c-existence} For small $\alpha, \epsilon$, there is a unique $c = c(\alpha,\epsilon)$ near $c_0 = U_0$ so that the linear dispersion \eqref{disp} holds. 
\end{lemma}

\begin{proof} The proof is straightforward. Indeed, let $F(c): = K_1 - K_3 + \cO(\alpha^2)$ so that the equation $F(c,\alpha, \epsilon) = 0$ is the relation \eqref{new-disp1}. We note that $F(c_0,0,0) = 0$ since $c_0 = U_0$. In addition, $\partial_c F(0,0,0) = -1 + \cO(\alpha^2)$, which is nonzero for sufficiently small $\alpha$. The existence of $c = c(\alpha,\epsilon)$ so that $F(c(\alpha,\epsilon),\alpha,\epsilon) = 0$ follows directly from the Implicit Function Theorem. \end{proof}


\subsection{Lower stability branch: $\alpha \approx R^{-1/7}$} 


Let us consider the case $\alpha  = A R^{-1/7}$, for some constant $A$. We recall that $\delta \approx (\alpha R)^{-1/3} = A^{-1/3} R^{-2/7}$. That is, $\delta \approx \alpha^2$ for the fixed constant $A$. By a view of \eqref{bound-zc}, we then have $|z_c| \approx C\delta$. More precisely, we have 
\begin{equation}\label{delta-zc} z_c/\delta \quad \approx\quad A^{4/3}.\end{equation}
Thus, we are in the case that the critical layer goes up to the boundary with $z_c/\delta$ staying bounded in the limit $\alpha,\epsilon \to 0$.  

We prove in this section the following lemma.

\begin{lemma}\label{lem-lowerbranch} Let $\alpha  = A R^{-1/7}$. For $R$ sufficiently large, there exists a critical constant $A_{c}$ so that 
the eigenvalue $c = c(\alpha,\epsilon)$ has its imaginary part changing from negative (stability) to positive (instability) as $A$ increases past $A = A_c$. In particular, 
$$
\I c \quad\approx \quad A^{-1/3} R^{-2/7}.
$$
\end{lemma}
\begin{proof} By taking the imaginary part of the dispersion relation \eqref{new-disp1} and using the bounds from Lemmas \ref{lem-ratioK1} and  \ref{lem-ratioK3}, we obtain 
\begin{equation}\label{disp3}(-1 + \cO(\alpha^2)) \I c + \cO(\alpha^4 \log \alpha) =\I\Big(  {\phi_{3}(0) \over \partial_z \phi_{3}(0)} \Big) = \cO(\delta (1+|z_c/\delta|)^{-1/2}).\end{equation}
which clearly yields $\I c = \cO(\delta (1+|z_c/\delta|)^{-1/2})$ and so $\I c \approx A^{-1/3} R^{-2/7}$. Next, also from Lemma \ref{lem-ratioK3}, the right-hand side is positive when $z_c/\delta$ is small, and becomes negative when $z_c/\delta \to \infty$.  Consequently, together with \eqref{delta-zc}, there must be a critical number $A_c$ so that for all $A > A_c$, the right-hand side is positive, yielding the lemma as claimed. \end{proof}

\subsection{Intermediate zone: $R^{-1/7} \ll \alpha \ll R^{-1/11}$}


Let us now turn to the intermediate case when 
$$
\alpha = A R^{-\beta}
$$
with $1/11 < \beta <  1/7$.
In this case
$\delta \approx \alpha^{-1/3} R^{-1/3} \approx A^{-1/3} R^{\beta/3 - 1/3}$ and hence
$\delta \ll \alpha^2$. That is,  the critical layer is away from the boundary: $\delta \ll z_c$ by a view of \eqref{bound-zc}. We prove the following lemma. 

\begin{lemma}\label{lem-midbranch} Let $\alpha  = A R^{-\beta}$ with $1/11<\beta<1/7$. For arbitrary fixed positive $A$, the eigenvalue $c = c(\alpha,\epsilon)$ always has positive imaginary part (instability) with   
$$
\I c \quad \approx\quad    A^{-3/2} R^{(3\beta - 1)/2}.
$$
\end{lemma}
\begin{proof}
As mentioned above, $z_c/\delta$ is unbounded in this case. Since $z_c \approx \alpha^2$ and $\delta \ll \alpha^2$, we indeed have 
$$z_c/\delta \quad \approx\quad  A^{7/3} R^{(1-7\beta)/3} \to \infty,$$
as $R \to \infty$ since  $\beta <1/7$. By Lemma \ref{lem-ratioK3}, we then have 
\begin{equation}\label{est-CAi}
\I (K_3 ) = \cO(\delta (1+|z_c/\delta|)^{-1/2})  \approx A^{-3/2} R^{(3\beta - 1)/2},
\end{equation}
and furthermore the imaginary of $K_3$ is positive since $z_c/\delta \to \infty$. It is crucial to note that in this case 
$$ \alpha^4 \log \alpha \approx R^{-4\beta} \log R,$$
which remains neglected in the dispersion relation \eqref{disp3} as compared to the size of the imaginary part of $K_3$, since $\beta >1/11$.

This yields the lemma at once. 
\end{proof}

\subsection{Upper branch instability: $\alpha \approx R^{-1/11}$}
Finally, let us study the upper branch case: $\alpha = A R^{-1/11}$. In this case, the term of order $\alpha^4\log\alpha$ is no longer neglected as compared to $K_3$ in the dispersion relation \eqref{disp3}. Precisely, we have 
$$
K_3  \approx A^{-3/2} R^{-4/11} ,\qquad  \alpha^4 \log \alpha \approx A^4 R^{-4/11} \log R.
$$
By a view of the linear dispersion relation just above the equation \eqref{new-disp}, the new dispersion relation now reads
$$\frac{\phi_1(0) - K_2\phi_2(0)}{\phi_1'(0) - K_2 \phi_2'(0)} = K_3 + \cO(\epsilon) = \cO(A^{-11/2}\alpha^4).$$
The left-hand side of \eqref{new-disp} consists precisely of the Rayleigh modes $\phi_{1,2}$, whereas the right-hand side can be neglected as compared to the $\alpha^4 \log\alpha$ terms. Since we have started with the case of the stable Rayleigh profiles, the corresponding eigenvalue is in the stable half-plane as $R\to \infty$. This shows that there must be a critical value $A_{2c}$ from instability (due to the previous case: $\alpha\ll R^{-1/11}$, or equivalently, $A\ll 1$) to stability, when $A$ increases past $A_{2c}$.

{\em Acknowledgement.} The authors would like to thank David G\'erard-Varet and Mark Williams for their many fruitful discussions and useful comments on an earlier draft of the paper. Guo and Nguyen's research is supported in part by NSFC grant 10828103 and NSF grant DMS-0905255, and by NSF grant DMS-1338643, respectively. Guo and Nguyen wish to thank Beijing International Center for Mathematical Research, and Nguyen thanks l'Institut de Math\'ematiques de Jussieu and ENS Lyon, for their support and hospitality in the summer of 2012 and 2013, during which part of this research was carried out.


\end{document}